\documentclass[a4paper,12pt,reqno]{amsart}

\usepackage[utf8]{inputenc}
\usepackage[T1]{fontenc}
\usepackage{lmodern}
\usepackage[english]{babel}
\usepackage{microtype}

\usepackage{amsmath,amssymb,amsfonts,amsthm,comment}
\usepackage{mathtools,accents}
\usepackage{mathrsfs}
\usepackage{aliascnt}
\usepackage{braket}
\usepackage{bm}
\usepackage{esint}
\usepackage{amsfonts}
\usepackage{csquotes}
\usepackage{tikz}

\usepackage[a4paper,twoside,top=0.85in, bottom=0.85in, left=0.7in, right=0.7in]{geometry} 
\usepackage{hyperref}
\usepackage[english,capitalize]{cleveref}

\usepackage{enumerate}
\usepackage{xcolor}
\usepackage{comment}

\usepackage{imakeidx} 
\usepackage{xparse} 
\usepackage{mathtools}
\usepackage{float}
\usepackage{comment}
\usepackage{nicefrac}
\usepackage{graphicx}
\usepackage{enumerate}
\usepackage{enumitem}

\definecolor{newblue}{rgb}{0.2, 0.3, 0.85}
\hypersetup{colorlinks=true, linkcolor=newblue, citecolor=newblue, urlcolor  = newblue}

\usepackage[maxbibnames=99,backend=biber, sorting=nyt, doi=false, url=false, isbn=false]{biblatex}
\makeatletter
\g@addto@macro\@floatboxreset\centering
\makeatother

\makeatletter
\def\newaliasedtheorem#1[#2]#3{
  \newaliascnt{#1@alt}{#2}
  \newtheorem{#1}[#1@alt]{#3}
  \expandafter\newcommand\csname #1@altname\endcsname{#3}
}
\makeatother

\numberwithin{equation}{section}

\newtheoremstyle{slanted}{\topsep}{\topsep}{\slshape}{}{\bfseries}{.}{.5em}{}

\theoremstyle{plain}
\newtheorem{theorem}{Theorem}[section]
\newaliasedtheorem{proposition}[theorem]{Proposition}
\newaliasedtheorem{question}[theorem]{Question}
\newaliasedtheorem{problem}[theorem]{Problem}
\newaliasedtheorem{lemma}[theorem]{Lemma}
\newaliasedtheorem{corollary}[theorem]{Corollary}
\newaliasedtheorem{counterexample}[theorem]{Counterexample}

\theoremstyle{definition}
\newaliasedtheorem{definition}[theorem]{Definition}
\newaliasedtheorem{example}[theorem]{Example}
\newaliasedtheorem{conjecture}[theorem]{Conjecture}

\theoremstyle{remark}
\newaliasedtheorem{remark}[theorem]{Remark}

\newcommand{\N}{\mathbb{N}}

\newcommand{\R}{\mathbb{R}}

\newcommand{\X}{{\rm X}}
\newcommand{\lip}{{\rm lip \,}}
\newcommand{\nchi}{{\raise.3ex\hbox{\(\chi\)}}}

\newcommand{\eps}{\varepsilon}

\let\phi\varphi

\newcommand{\abs}[1]{\left\lvert#1\right\rvert}

\newcommand{\st}{\ensuremath{\ :\ }} 
\newcommand{\eqdef}{\ensuremath{\vcentcolon=}}

\DeclareMathOperator{\Lip}{Lip}

\newcommand{\haus}{\mathcal{H}}

\newcommand{\dist}{\mathsf{d}}

\newcommand{\meas}{\mathfrak{m}}

\newcommand{\vol}{\mathcal{H}^n}

\DeclareMathOperator{\esssup}{ess sup}
\newfont{\tmpf}{cmsy10 scaled 2500}

\newcommand{\de}{\ensuremath{\,\mathrm d}} 

\author[G. Antonelli]{Gioacchino Antonelli}
\address{Department of Mathematics, University of Notre Dame, Hurley Hall, 255 Hurley, Notre Dame, IN 46556, United States}
\email{gantonel@nd.edu}

\author[M. Fogagnolo]{Mattia Fogagnolo}
\address{Dipartimento di Matematica Tullio Levi-Civita, Universit\`a di Padova, via Trieste 63, 35010 Padova, Italy}
\email{mattia.fogagnolo@unipd.it}

\author[S. Nardulli]{Stefano Nardulli}
\address{Centro de Matemática Cognição e Computação, Universidade Federal do ABC, Av. Dos Estados 5001, Santo André, SP, Brazil}\email{stefano.nardulli@ufabc.edu.br}

\author[M. Pozzetta]{Marco Pozzetta}
\address{Dipartimento di Matematica, Politecnico di Milano, Via Bonardi 9, 20133 Milano, Italy}
\email{marco.pozzetta@polimi.it}


\begin{document}
\title[Positive mass and isoperimetry for continuous metrics]{Positive mass and isoperimetry for continuous metrics with nonnegative scalar curvature}

\medskip

\begin{abstract}
This paper deals with quasi-local isoperimetric versions of the positive mass theorem on $3$-manifolds endowed with continuous complete metrics having nonnegative scalar curvature in a suitable weak sense. As a corollary, we derive existence results for isoperimetric sets in such low regularity setting. Our main tool is a new local version of the weak inverse mean curvature flow enjoying $C^0$-stable quantitative estimates.
\end{abstract}

\maketitle

\tableofcontents

\vspace{-0.7cm}

\section{Introduction}

\subsection{Main results}
The classical positive mass theorem (PMT) in $3$ dimensions, originally due to Schoen--Yau \cite{SchoenYauPMT}, asserts that the ADM mass of a $3$-dimensional smooth complete asymptotically flat manifold with nonnegative scalar curvature is nonnegative. This seminal result has been later rediscovered, generalized, improved and exploited in a number of ways.  Without any attempt to be complete, we refer to \cite{WittenPMT,  SYDefinitivePMT, AMO21, bray-karazas-khuri-stern} for different proofs and extensions to higher dimensions, and to \cite{HuiskenIlmanen, BrayPenrose} for its main refinement, i.e., the Riemannian Penrose inequality.

In this paper we prove  quasi-local versions of the PMT for continuous metrics with nonnegative scalar curvature in the approximate sense on $3$-manifolds. We say that $(M,g)$ is a \textit{$C^0$-Riemannian manifold} if $M$ is a smooth differentiable manifold, and $g$ is a $C^0$-metric on $M$.
The following is the weak notion of nonnegative scalar curvature we adopt in our work.

\begin{definition}[$R_g\geq 0$ in the approximate sense]\label{def:NonnegativeScalInApproximate}
    Let $(M,g)$ be a complete $C^0$-Riemannian manifold without boundary, and let $\Omega\subset M$ be an open set. We say that \textit{$R_{g}\geq 0$ in the approximate sense on $\Omega$} if there exist
    smooth complete Riemannian metrics $g_j$ on $M$, such that: 
    \begin{enumerate}
        \item $g_j$ converges to $g$ locally uniformly on $M$;
        \item There exists a sequence $(\varepsilon_j)_{j\in\mathbb N}$ of positive numbers such that $\varepsilon_j\to 0$, and $R_{g_j}\geq -\varepsilon_j$ on $\Omega$.
    \end{enumerate}
\end{definition}
By a result of Gromov \cite{GromovScal} (see also \cite{BamlerGromov} for a different proof using Ricci flow) a smooth Riemannian manifold has nonnegative scalar curvature if and only if $R_g\geq 0$ holds in the approximate sense on $M$.

For a Borel set $E\subset M$, we let $|E|,P(E)$ denote the volume and the perimeter of $E$ respectively, see \cref{sec:Perimeter} for precise definitions. Let us recall that the perimeter of $E$ never charges the boundary of $M$, if any. The following quantity is a localized version of the \emph{isoperimetric mass} introduced by Huisken \cite{Huisken}. 

\begin{definition}[Quasi-local isoperimetric mass]\label{def:IsopMass}
    Let $(M, g)$ be a $C^0$-Riemannian $3$-manifold, possibly with
boundary. Then, for an open set $E\subset M$ we
define the \textit{quasi-local (isoperimetric) mass} of $E$ as
\begin{equation}\label{eqn:mQL}
\mathfrak{m}_{\mathrm{QL}}(E):=\frac{2}{P(E)}\left(|E|-\frac{P(E)^{3/2}}{6\sqrt \pi}\right).
\end{equation}
\end{definition}
Huisken's isoperimetric mass is defined by
\begin{equation}
\label{eq:isomass}
\begin{split}
\mathfrak{m}_{\mathrm{iso}}\eqdef \sup
\bigg\{
\limsup_{j\to+\infty} \mathfrak{m}_{\mathrm{QL}}(E_j)
\,\,\st\,\, 
&E_j\subset M, P(E_j)<+\infty \,\,\forall\, j,\, P(E_j)\xrightarrow[]{}+\infty
\bigg\}.    
\end{split}
\end{equation}
For asymptotically flat $3$-manifolds with nonnegative scalar curvature, the isoperimetric mass coincides with the ADM mass {whenever the boundary is minimal} \cite[Theorem 3]{JaureguiLee}, \cite[Theorem C.2]{ChodoshEichmairShiYu21} (see also \cite[Theorem 1.4]{BFMIsoperimetricRiemannianPenrose} for a proof in the sharp asymptotic regime).  

Huisken has conjectured \cite{Huisken2} that a weak (isoperimetric) PMT should hold true for continuous Riemannian metrics if one interprets the notion of nonnegative scalar curvature (shortly, $R_g\geq 0$) in an appropriate weak sense. On the other hand in \cite{JaureguiLeeUnger24IsoperimetricMass}, after the first version of this preprint appeared, the authors proved that the property of being $C^0_{\rm loc}$-asymptotic to $\mathbb R^n$ (see \cref{def:C0locasymptotic}) is enough to imply $\mathfrak{m}_{\mathrm{iso}}\geq 0$ in every dimension, without further imposing conditions on the curvature. 
The following is the precise meaning of  $C^0_{\rm loc}$-asymptoticity.
\begin{definition}\label{def:C0locasymptotic}
    Let $K\subset M$ be a compact set (possibly empty) of a $C^0$-Riemannian manifold $(M,g)$. We say that an unbounded connected component $\mathcal{E}$ of $M\setminus K$ is \textit{$C^{0}_{\mathrm{loc}}$-asymptotic to a $C^0$-Riemannian manifold $(N,h)$} if the following holds. For every diverging sequence $\mathcal{E}\ni p_i\to +\infty$ there exists a point $o\in N$ such that $(M,g_i,p_i)\to (N,h,o)$ in the $C^0$-sense, see \cref{def:C0locconvergence}.
\end{definition}

In this paper, we prove that in $3$-manifolds that are $C^0_{\rm loc}$-asymptotic to $\mathbb{R}^3$, the curvature condition $R\geq 0$ in fact manifests itself at large scales with the existence of arbitrarily large open sets $E$ with $\mathfrak{m}_{\mathrm{QL}}(E)\geq 0$.

\begin{theorem}\label{thm:CorMassa}
Let $(M,g)$ be a complete 3-dimensional $C^0$-Riemannian manifold without boundary. Let $K\subset M$ be a compact set, and let $\mathcal{E}$ be an unbounded connected component $M\setminus K$. Assume that $\mathcal{E}$ is $C^0_{\mathrm{loc}}$-asymptotic to $\mathbb R^3$, see \cref{def:C0locasymptotic}, and that $R_{g}\geq 0$ in the approximate sense on $\mathcal{E} \setminus K'$, see \cref{def:NonnegativeScalInApproximate}, where $K'\subset M$ is a compact set.
Then for any ${C}>0$, there exists $E \subset \mathcal{E} \setminus K'$ such that
\begin{equation}
\label{eq:quasilocalbig}
\min\{P(E),|E|\}\ge {C} ,
\qquad
\mathfrak{m}_{\rm QL}(E) \ge 0.
\end{equation}
\end{theorem}

Of course \eqref{eq:quasilocalbig} directly implies $\mathfrak{m}_{\mathrm{iso}}\geq 0$. 
More importantly, while the latter has been recently shown to hold in asymptotically flat manifolds regardless of any curvature condition \cite{JaureguiLeeUnger24IsoperimetricMass}, it is not difficult to construct asymptotically flat $3$-manifolds $(M, g)$ such that $\mathfrak{m}_{\rm QL}(E) < 0$ for any bounded $E \subset M$. A general class of such manifolds is given by Cartan--Hadamard asymptotically flat $3$-manifolds with strictly negative sectional curvatures. In fact, any bounded domain sitting in such a space is known to satisfy the sharp Euclidean isoperimetric inequality \cite{kleiner-isoperimetric} with strict inequality.

\smallskip

The same strategy used to prove \cref{thm:CorMassa} also yields a quasi-local isoperimetric PMT for \emph{small} sets, without any asymptotic condition.

\begin{theorem}
\label{thm:small}
Let $(M,g)$ be a complete 3-dimensional $C^0$-Riemannian manifold without boundary. Let $\Omega \subset M$ be an open subset such that $R_g\ge0$ in the approximate sense on $\Omega$.
Then for any $o \in \Omega$ there exists $r_o>0$ such that for any $r \in (0,r_o]$ there exists $E \subset B_r(o)$ such that $\mathfrak{m}_{\rm QL}(E) \ge 0$.
\end{theorem}

As already suggested by Huisken \cite{Huisken2}, the nonnegativity of the isoperimetric mass should be compared with the following asymptotic formula, that can be directly deduced from \cite[Theorem 3.1]{gray}: there exists a positive constant $c$ such that for any smooth $3$-dimensional Riemannian manifold $(M, g)$, and any $o \in M$, it holds 
\begin{equation}
\label{eq:asy-scalar}
{R_g}(o) =  \lim_{r \to 0^+} \frac{c}{r^{5}} \left(\abs{B_r(o)} - \frac{P(B_r(o))^{3/2}}{6\sqrt{\pi}} \right).
\end{equation}
The latter \eqref{eq:asy-scalar} implies that, if ${R}_g(o) > 0$, then $\mathfrak{m}_{\mathrm{QL}} (B_r(o)) > 0$ for sufficiently small $r$. However, this asymptotic reasoning does not provide a method to construct sets with nonnegative quasi-local mass in a smooth manifold with nonnegative scalar curvature, let alone sets {for which the condition $\mathfrak{m}_{\mathrm{QL}} \geq 0$ remains stable} under $C^0$-limits. This is what \cref{thm:small} instead provides in dimension $3$. Taking into account \eqref{eq:asy-scalar} 
we wonder whether the existence of arbitrarily small sets with nonnegative quasi-local mass could in fact characterize nonnegative scalar curvature. If this were the case, it would yield a new characterization of nonnegative scalar curvature that by \cref{thm:small} would be stable with respect to $C^0$-limits of smooth manifolds.

Both \cref{thm:CorMassa} and \cref{thm:small} are new even in the smooth setting, where they were known only under additional asymptotic assumptions, ensuring the existence of a global, proper weak inverse mean curvature flow \cite{HuiskenIlmanen}. For general conditions under which such a flow exists, see \cite{KaiXu}.
In this case, the sets with nonnegative quasi-local mass are given by domains evolving through the inverse mean curvature flow, as a consequence of a reverse isoperimetric inequality one could infer from the work of Shi \cite{Shi}. Similarly, our proof will rely on both a new localized version of the weak inverse mean curvature flow, thus allowing to drop the asymptotic assumptions, and on quantitative estimates that are stable for $C^0$-limits, under a lower bound on the scalar curvature. 
\smallskip

The nonnegativity of the isoperimetric quasi-local mass for both small and large sets also has implications on the existence of isoperimetric regions. We show that in the $C^0$-setting with nonnegative scalar curvature, isoperimetric sets exist for sequences of both small and large volumes. The result below should be compared with \cite[Proposition K.1]{CarlottoChodoshEichmair}.

\begin{theorem}\label{thm:CorIsop}
Let $(M,g)$ be a complete 3-dimensional $C^0$-Riemannian manifold without boundary, and assume that $R_{g}\geq 0$ in the approximate sense on $M\setminus \mathcal{C}$, where $\mathcal{C}$ is a compact set, see \cref{def:NonnegativeScalInApproximate}. Assume in addition that $M$ is $C^0_{\mathrm{loc}}$-asymptotic to $\mathbb R^3$, see \cref{def:C0locasymptotic}.

Then there exists on $M$ a sequence of isoperimetric sets $(E_j)_{j\in\mathbb N}$ such that $|E_j|\to +\infty$ and a sequence of isoperimetric sets $(F_j)_{j\in\mathbb N}$ such that $|F_j|\to 0$. 
\end{theorem}

Very much like \cref{thm:CorMassa}, also \cref{thm:CorIsop} strongly weakens the asymptotic assumptions for the  existence of isoperimetric sets even in the smooth setting.
The literature on the subject is very vast, in particular in relation to the study of canonical foliations of stable CMC compact hypersurfaces:
we refer to the seminal \cite{HuiskenYauFoliations} and to \cite{CarlottoChodoshEichmair, eichmair-metzger_jdg, EichmairMetzger, nerz, ChodoshEichmairShiYu21, yu, eichmair-koerber, sinestrari-tenan}, as well as to the references therein, for a fairly complete picture.  

In the setting of \cref{thm:CorIsop} it is an interesting open problem to analyze existence of isoperimetric sets for \textit{any}  volume, uniqueness or foliation properties of such isoperimetric sets, as in \cite{nerz, ChodoshEichmairShiYu21, yu} (see also the survey \cite{BFSurvey}). Under the additional assumption that the isoperimetric profile $I$ is strictly increasing, we can actually prove existence of isoperimetric sets for every volume, see \cref{prop:profileincreasing}.

\bigskip

\subsection{Strategy}\label{sec:Strategy}
We briefly discuss the strategy of the proof of \cref{thm:CorMassa}, and \cref{thm:CorIsop}, referring the reader to \cref{sec:IMCFWithGlobalIsoperimetric}, and \cref{sec:Proofs} for the details. 

The starting point in the proof of \cref{thm:CorMassa} and \cref{thm:small} is the following consequence of a result due to Shi \cite{Shi}, after important insights by Brendle--Chodosh \cite{brendle-chodosh}: in a smooth $3$-dimensional asymptotically flat manifold with nonnegative scalar curvature the level sets of the weak inverse mean curvature flow (shortly, IMCF, see \cref{def:WeakIMCFLevelSet}) issuing from a point satisfy a reverse Euclidean isoperimetric inequality with the sharp constant, as long as their boundaries are connected.
Hence the isoperimetric deficit in \eqref{eqn:mQL} is nonnegative when computed on such sets, and this directly provides examples of arbitrarily large regions with nonnegative quasi-local mass.

We will push this heuristic to show that on the approximating manifolds $(M,g_j)$ one can define a well-behaved local weak IMCF $w_j$ on punctured balls $B_j$. This is done by taking scaled limits of the logarithms of $p$-Green functions on such balls, for $p\to 1^+$, as pioneered by R. Moser \cite{MoserIMCF}. In fact, we build on the sharp gradient estimate obtained in \cite{kotschwar_localgradientestimatesharmonic_2009}, and on the Harnack inequality with explicit constants in \cite{salvatori-rigoli-vignati} to get such functions $w_j$ with bounds from below only in terms of constants that stay bounded in $j$ when the $B_j$'s are $C^0$-close enough to a Euclidean ball. In general, these estimates will be in force for balls $B_j$ that are small enough, while under additional $C^0_{\rm loc}$-asymptotics, they can be obtained for arbitrarily large balls $B_{R_j}(p_j)$ with $R_j \to \infty$ and $p_j \to \infty$.  This is in fact the content of the quantitative existence result for local IMCFs  \cref{thm:existenceimcfisoplocalized}. The latter finally allows to pass to the limit the level sets of these IMCFs, obtaining the sets satisfying the (sharp) reverse Euclidean isoperimetric inequality (or equivalently with nonnegative quasi-local mass) claimed in \cref{thm:CorMassa} and \cref{thm:small}, respectively. 
\medskip

The proof of the existence of isoperimetric sets in \cref{thm:CorIsop} is not constructive, and it is based on a contradiction argument. 
One notices that if 
after a certain volume threshold isoperimetric sets do not exist,
then the isoperimetric profile is strictly increasing for large volumes: this is a consequence of a generalized existence theorem for the isoperimetric problem, see \cref{thm:MassDecompositionC0}.
In addition, arguing as in the previous paragraph, we construct sets that satisfy the reverse sharp Euclidean isoperimetric inequality with arbitrarily large volumes and perimeters, and that avoid any fixed compact set. This is enough to show, again using \cref{thm:MassDecompositionC0} and that the isoperimetric profile is increasing, that isoperimetric sets with arbitrarily large volumes must exist, thus resulting in a contradiction.
The analogous statement for small volumes follows from an easier instance of this reasoning.

\subsection{Comments and comparison with related literature}

We collect here some comments and perspectives on the main results of this paper.

\subsubsection{Other notions of weak scalar curvature bounds, and relations with the works \cite{PaulaGAFA, PaulaADM}}\label{NotionsOfWeakScalarBounds}
    Several notions of scalar curvature lower bounds for smooth manifolds endowed with $C^0$-Riemannian metrics have been proposed in the recent years. In \cite{GromovScal} Gromov has proposed a definition based on {nonexistence of suitable} small polyhedra on the manifold, see the recent works \cite{ChaoLiPolyedron, BrendleInv} motivated by this study; a definition  based on regularization through Ricci flow has been suggested by Burkhardt-Guim in \cite{PaulaGAFA}; assuming the nonnegativity of the right hand side of \eqref{eq:asy-scalar} as a replacement for nonnegative scalar curvature has been hinted at by Huisken \cite{Huisken2}. Let us compare now nonnegative scalar curvature in the approximate sense  with the notion in \cite{PaulaGAFA, PaulaADM}. 
    
  Assume that a complete $C^0$-Riemannian manifold $(M,g)$ is \textit{globally} $C^0$-asymptotic to $\mathbb R^3$ outside a compact set $K\subset M$, i.e. $M \setminus K = \R^n \setminus B$ for some ball $B$ and $g = \delta + o(1)$ at infinity, with $\delta$ the flat metric. Hence, for $\beta\in (0,1/2)$, we claim that if $R_{g}\geq 0$ in the $\beta$-weak sense \cite[Definition 2.3]{PaulaADM} on $M\setminus K$, then $R_g\geq 0$ in the approximate sense on $M\setminus K$, up to possibly enlarging $K$. Indeed, this is due to the fact that under the $C^0$-asymptotic hypothesis at infinity one can first define a Ricci-deTurck flow $(M,g_t)$ starting from $g$ which $C^0$-converges to $(M,g)$ locally uniformly by using \cite[Theorem 1.1]{SimonC0}. Then, since $R_{g}\geq 0$ in the $\beta$-weak sense on $M\setminus K$, \cite[Lemma 5.1]{PaulaADM} implies that, up to possibly enlarging $K$, $R_{g_t}\geq -o(1)$ on $M\setminus K$ as $t\to 0$, which in turn implies that $R_{g}\geq 0$ in the approximate sense on $M\setminus K$ by definition.
  In the case $M$ is a compact manifold, the previous reasoning has been explicitly recorded in \cite[Corollary 1.6]{PaulaGAFA}. {Related to this, it would be interesting to understand whether a non-compact $C^0$-manifold $(M,g)$ that has $R_g\geq 0$ in the approximate sense on $M$ admits a smooth metric $\tilde g$ with $R_{\tilde g}\geq 0$. The compact case has been settled in \cite[Corollary 1.8]{PaulaGAFA}. Taking into account the discussion above, this seems likely to hold when $(M,g)$ is, in addition, $C^0$-asymptotic to $\mathbb R^n$ globally. 

    As a consequence of the above discussion, under the stronger assumption that $(M,g)$ is globally $C^0$-asymptotically flat one gets that our theorems \ref{thm:CorMassa}, \ref{thm:small}, and \ref{thm:CorIsop} hold under the assumption that $R_{g}\geq 0$ in the $\beta$-weak sense outside a compact set, for some $\beta\in (0,1/2)$ (\cite[Definition 2.3]{PaulaADM}). 

    Finally, we point out that as of today it is not known whether the definitions in \cite{GromovScal} and \cite{PaulaGAFA} are equivalent, compare with the discussion \cite[Page 1707-1708]{PaulaGAFA}, nor any relation with Huisken's notion in \cite{Huisken2} has  been investigated yet.  

\subsection{The question of rigidity} 
In the smooth global PMT \cite{SchoenYauPMT}, one gets that the mass is zero if and only if the metric is flat. Restricting the attention to the smooth case, an easy combination of the local IMCF arguments leading to \cref{thm:small} with the rigidity triggered by the constancy of the Hawking mass along the weak IMCF \cite{HuiskenIlmanen} allows us to prove the following quasi-local counterpart.  
\begin{theorem}
\label{thm:quasilocalrigidity}
Let $(M, g)$ be a complete $3$-dimensional smooth Riemannian manifold. Suppose that ${R}_g \geq 0$ on an open set $\Omega \subset M$. If
\begin{equation}
\label{eq:quasilocalrigidity}
\sup \{\mathfrak{m}_{\rm QL} (E) \, | \, E \Subset \Omega\} \leq 0,
\end{equation}
then $\Omega$ is flat.
\end{theorem}
The result will follow from observing that \eqref{eq:quasilocalrigidity} yields 
that 
\begin{equation}\label{eqn:Hawking0}
\mathfrak{m}_{H} (\partial E_t) = 0
\end{equation}
for all the domains $\partial E_t \Subset \Omega$ evolving through the local IMCF, where $\mathfrak{m}_{H} (\partial E_t)$ denotes the Hawking mass \eqref{eqn:HawkingMass} of their boundary. We can then easily conclude invoking a rigidity argument by Huisken--Ilmanen \cite{HuiskenIlmanen}.
We observe that a counterpart of \cref{thm:quasilocalrigidity}, with \eqref{eq:quasilocalrigidity} replaced by
\begin{equation}
\label{eq:rigidity-hawking}
    \sup \{\mathfrak{m}_{H} (\partial E) \, | \, E \Subset \Omega \, \text{with } \partial E \in C^2\}  \leq 0,
\end{equation}
has been already obtained by Mondino--Templeton-Browne by different means \cite{t.browne-mondino}. In fact, we stress that having proved the existence of a weak local IMCF allows us to derive a different proof of the result in \cite{t.browne-mondino}. Indeed, \eqref{eq:rigidity-hawking} would directly imply \eqref{eqn:Hawking0}, thus allowing us to conclude using the rigidity argument of Huisken--Ilmanen as mentioned above, see also the proof of \eqref{eq:quasilocalrigidity} below.
Notice that the condition \eqref{eq:quasilocalrigidity} makes sense also for $C^0$-metrics.

We can then pose the following problem
which is likely to require the use of new techniques tailored for the $C^0$-setting, or some refined approximation procedure.

\begin{question}\label{prob:Rigidity}
Let $(M, g)$ be a complete $C^0$-manifold of dimension $3$ without boundary. Assume that $M$ has $R_g\geq 0$ in the approximate sense  in an open set $\Omega \subset M$, and that \eqref{eq:quasilocalrigidity} holds. Is it true that $g$ is smooth and flat on $\Omega$?
\end{question}
\medskip
\medskip

\textbf{Acknowledgments}. A part of this research was carried out while G.A., M. F., and M.P. were visiting S.N. at UFABC. The joyful working conditions and the hospitality are gratefully acknowledged. Another part of the work was carried out while G. A. and M. F. were participating in the workshop ``Recent Advances in Comparison Geometry'' in Hangzhou. They thank the Banff International Research Station and the Institute of Advanced Studies in Mathematics for this enriching opportunity. G.A. is grateful to Paula Burkhardt-Guim for useful discussions around the topic of the paper.
M. F. is grateful to Andrea Mondino for his interest and for having brought to his attention the paper \cite{t.browne-mondino}.
M.P. is grateful to Roberto Ladu for clarifying discussions and stimulating advice. The authors are grateful to Luca Benatti, Luciano Mari, and Kai Xu for useful discussions around the topic of this paper. They are also grateful to Virginia Agostiniani, Alessandro Carlotto, Jeff Jauregui, Dan Lee, Lorenzo Mazzieri, and Christina Sormani for their interest in this work.

G.A. has been partially supported by the AMS-Simons Travel Grant, and the NSF DMS Grant No. 2505713.
M.F. is supported by the Project ``GIANTS", funded by INdAM.
M. F. has been supported by the Project ``iNEST: Interconnected Nord-Est Innovation Ecosystem'' funded under the National Recovery and Resilience Plan (NRRP). M. F. and M. P. are members of INdAM-GNAMPA. S.N. was supported by FAPESP Auxílio Jovem Pesquisador \# 2021/05256-0, FAPESP Sprint 	\# 23/08246-0, CNPq  Bolsa de Produtividade em Pesquisa 1D \# 12327/2021-8, 23/08246-0, Geometric Variational Problems in Smooth and Nonsmooth Metric Spaces \# 441922/2023-6. M.P. is partially supported by the PRIN Project 2022E9CF89 - PNRR Italia Domani, funded by EU Program NextGenerationEU.

\section{Preliminaries}

We gather in this section a number of fundamental properties of $C^0$-Riemannian manifolds, possibly endowed with $C^0$-asymptotics. 

\subsection{Basic definitions and properties of $C^0$-metrics}

\begin{definition}[$C^0$-Riemannian manifold]
A \textit{$C^0$-Riemannian manifold} is a couple $(M,g)$, where $M$ is a smooth $n$-dimensional differentiable manifold (possibly with boundary) and $g$ is a continuous metric tensor. More precisely, in any local chart $(U,\varphi)$ on $M$, the metric tensor in local coordinates is represented by a symmetric positive definite matrix with components $g_{ij} \in C^0(\varphi(U))$.
\end{definition}

We briefly recall some metric properties of $C^0$-Riemannian manifolds $(M,g)$, and we refer the reader to \cite{burtscher-continuousmetrics, petrunin-metricgeometrytwolectures} for a more detailed discussion. As in the smooth case, the continuous Riemannian metric $g$ defines a length structure on absolutely continuous curves on $M$, which in turn gives raise to a distance $\dist$ that induces the manifold topology. We always assume that $(M,\dist)$ is a complete metric space, so that by Hopf--Rinow Theorem the distance $\dist$ is geodesic, and every closed ball is compact. 

The volume form induced by $g$ in a local chart induces integration with respect to $\mathrm{vol}:=\sqrt{\mathrm{det}\,g} \,\mathcal{L}^n$, where $\mathcal{L}^n$ denotes Lebesgue measure. It can be proved that $\mathrm{vol}=\mathcal{H}^n$, where $\mathcal{H}^n$ is the $n$-dimensional Hausdorff measure relative to $\dist$, and that $\mathcal{H}^n$ is a Radon measure. Hence $(M,\dist,\mathcal{H}^n)$ is a complete and separable metric measure space. We will often denote $|E|:=\mathcal{H}^n(E)$. In the following, if $E$ is a measurable set, integrals over $E$ are tacitly understood to be taken with respect to $\haus^n$.

Given $f\in C^{\infty}(M)$, one can define $\nabla f$ as in the smooth case by setting $g(\nabla f,X)=\mathrm{d}f(X)$ for every vector field $X$ on $M$. Hence, in a local chart $\{\partial_i\}_{i=1,\ldots,n}$, we have $|\nabla f|^2:=g(\nabla f,\nabla f)=g^{ij}\partial_i f\partial_j f$. Notice that by the classical Rademacher Theorem applied in chart, if $f\in \mathrm{Lip}_{\mathrm{loc}}(M)$, then $\nabla f$ exists $\mathcal{H}^n$-almost everywhere.

From now on we will always assume that $C^0$-Riemannian manifolds are complete. We recall that a map between metric spaces $F:(X,\dist_X)\to (Y,\dist_Y)$ is said to be \textit{$L$-biLipschitz}, with $L\geq 1$, when $L^{-1}\dist_X(a,b)\leq \dist_Y(F(a),F(b))\leq L\dist_X(a,b)$ for every $a,b\in X$.

\begin{lemma}\label{lem:MetricheVicineImplicaBilipschitz}
    Fix $n\in\N$ with $n\ge2$. Then for any $\delta>0$ there exists $\eps>0$ such that the following holds.
    Let $(M,g)$ be a $C^0$-Riemannian manifold and let $o\in M$ and $R>0$. Denote by $\dist$ the Riemannian distance on $M$. Let $N$ be a differentiable manifold and let $\Omega\subset N$ be an open set. Suppose that there exists a diffeomorphism $F:B_{10R}(o)\to \Omega$ and a $C^0$-Riemannian metric $h$ on $\Omega$ such that $|(F^* h - g)(v,v)| \le \eps g(v,v)$ for any $v \in T_xM$ and any $x \in B_{10R}(o)$.
    
    Then, letting \[\tilde\dist(x,y) \eqdef \inf\left\{ \int_0^1 |\gamma'|_h \st \gamma:[0,1]\to\Omega, \gamma(0)=x, \gamma(1)=y  \right\},\] the map $F|_{B_R(o)}:(B_R(o), \dist) \to (F(B_R(o)), \tilde\dist)$ is $(1+\delta)$-biLipschitz with its image.
\end{lemma}

\begin{proof}
    Denote $o'\eqdef F(o)$. Let $x,y \in B_R(o)$. A constant speed geodesic $\gamma:[0,1]\to M$ from $x$ to $y$ has image contained in $B_{5R}(o)$. Hence we can estimate
    \[
    \tilde\dist(F(x),F(y)) \le \int_0^1 h\left( (F\circ\gamma)', (F\circ\gamma)' \right)^{\frac12} \de t \le \sqrt{1+\eps} \int_0^1 g(\gamma',\gamma')^{\frac12} \de t \le \sqrt{1+\eps} \,\dist(x,y).
    \]
    Denoting $G=F^{-1}$, by assumptions we have that 
    \[
    \left|G^*g(w,w)-h(w,w) \right| \le  \eps G^*g(w,w),
    \]
    for any $w\in T_zN$ and $z \in \Omega$.  
    Taking now $p,q \in F(B_R(o))$ and a constant speed curve $\sigma:[0,1]\to \Omega$ from $p$ to $q$ such that $\int_0^1 |\sigma'|_h \le (1+\eta)\tilde\dist(p,q)$ for $\eta\in(0,1)$, we can similarly estimate
    \[
    \dist(G(p),G(q)) \le \int_0^1 g( (G\circ \sigma)', (G\circ \sigma)')^{\frac12} \de t \le \frac{1}{\sqrt{1-\eps}} \int_0^1 h(\sigma',\sigma')^{\frac12} \de t  \le \frac{1+\eta}{\sqrt{1-\eps}} \tilde\dist(p,q).
    \]
    Sending $\eta\to0$, for $\eps$ small enough the claim follows.
\end{proof}

\begin{definition}[$C^0$-convergence]\label{def:C0locconvergence}
We say that pointed $C^0$-Riemannian manifolds \textit{$(M_i,g_i,p_i)$ $C^{0}$-converge to $(M,g,p)$} if the following holds. For every $R,\varepsilon>0$ there exist $i_0:=i_0(R,\varepsilon)$ and an open set $\Omega\subset  M$ such that we have:
\begin{enumerate}
\item $\overline{B}_R(p)\subset  \Omega$;
\item for every $i\geq i_0$ there exists an embedding $F_i:\Omega\to M_i$ such that 
\begin{itemize}
\item $F_i(p)=p_i$;
\item $\overline{B}_R(p_i)\subset  F_i(\Omega)$;
\item $ |(F_i^*g_i-g)_x(v,v)|\leq\varepsilon g_x(v,v)$ for every $x\in \overline{B}_R(p)$ and $v\in T_xM$.
\end{itemize}
\end{enumerate}
\end{definition}

\begin{remark}\label{rem:BiLipschitz}
In the notation of \cref{def:C0locconvergence}, arguing as in the proof of \cref{lem:LocalePiattezzaC0metrics}, it follows that for any $R,\eps>0$ the map $F_i:(B_R(p),\dist)\to (F_i(B_R(p)), \dist_i)$ is $(1+\eps)$-biLipschitz with its image for any $i$ large enough.
\end{remark}

\begin{lemma}\label{lem:LocalePiattezzaC0metrics}
Let $(M,g)$ be a $C^0$-Riemannian manifold, let $x_0 \in M\setminus\partial M$, and denote by $\dist$ the Riemannian distance on $M$. Then the following holds.
\begin{itemize}
    \item The metric tangent space of $M$ at $x_0$ is isometric to the Euclidean space $\R^n$, i.e., the rescalings $(M, \delta^{-2}g, x_0)$ converges to $(\R^n, g_{\rm eu}, 0)$ in $C^0$-sense as $\delta\searrow0$. 

    \item For any $\delta>0$ there is $r=r(x_0,\delta)>0$ and a local chart $\varphi: (B_r(x_0),\dist) \to (\varphi(B_r(x_0)), \dist_{\rm eu}) \subset \R^n$ such that $\varphi$ is $(1+\delta)$-biLipschitz with its image, where $\dist_{\rm eu}$ denotes Euclidean distance.
\end{itemize}
\end{lemma}

\begin{proof}
Passing in local coordinates we can identify a neighborhood of $x_0$ in $M$ with $(\Omega, g)$, where $\Omega\subset \R^n$ is open and $g=(g_{ij})$ is the metric in local coordinates. Also we can identify $x_0$ with the origin $0 \in \Omega$. Let $r_k\searrow0$. For the first part of the statement, it is sufficient to prove that $(\Omega, r_k^{-2} g, 0)$ $C^0_{\rm loc}$-converges to $(\R^n, g_0, 0)$, where $g_0$ is the constant metric given by $g$ evaluated at the origin. The diffeomorphism $\Phi_k: (r_k^{-1}\Omega, g_0) \to (\Omega, r_k^{-2}g)$ given by $\Phi_k(x)=r_k\, x$ satisfies $(\Phi_k^*(r_k^{-2}g))_x(v,v) = g_{r_k \, x}(v,v)$ for any tangent vector $v$. Hence, given any $R>0$ and $\eps>0$, for $k$ large enough we have that
\[
1-\eps \le \frac{(\Phi_k^*(r_k^{-2}g))_x(v,v)}{g_0(v,v)} \le 1+\eps,
\]
for any $x \in B_R^{g_0}(0) \subset (r_k^{-1}\Omega, g_0)$ and any tangent vector $v\neq0$. This establishes the convergence in $C^0$ to $\R^n$ as $k\to\infty$.

Concerning the second part of the statement, let $r>0$ to be chosen small and fix a local chart $\varphi:B_{10r}(x_0) \to \varphi(B_{10r}(x_0)) =: \Omega\subset \R^n$ such that $\varphi(x_0)=0$, and denote $g_0:= g_{\rm eu}$. We identify $B_{10r}(x_0)$ with $(\Omega, g)$ in the local chart. Let $\tilde \dist$ be as in \cref{lem:MetricheVicineImplicaBilipschitz} with $N=\R^n$. By continuity of the metric tensor, for any $\eps>0$ we can take $r$ so small that $|g_x(v,v)- g_0(v,v)| \le \eps g_0(v,v)$ for any $x \in \Omega$ and any $v \in \R^n$. For $\eps$ small enough, it follows from \cref{lem:MetricheVicineImplicaBilipschitz} that the identity ${\rm id}:(B^g_r(0), \dist) \to (B^g_r(0), \tilde\dist)$ is $(1+\delta)$-biLipschitz. It remains to observe that $\tilde\dist=\dist_{\rm eu}$ on $B^g_r(0)$. Indeed, fix $p,q \in B^g_r(0)$ and let $\gamma_i:[0,1]\to \Omega$ be a sequence of constant speed curves such that $\int_0^1 |\gamma_i'|_{\rm eu} \to \tilde\dist(p,q)$. It follows that $\gamma_i([0,1]) \subset B^g_{9r}(0)$ for large $i$. Indeed, taking $\sigma:[0,1]\to B^g_{5r}(0)$ a constant speed geodesic for $\dist$ from $p$ to $q$, we know that
\[
\tilde\dist(p,q) \le \int_0^1 |\sigma'|_{\rm eu} \le \frac{1}{\sqrt{1-\eps}}\int_0^1 |\sigma'|_g = \frac{1}{\sqrt{1-\eps}} \dist(p,q) \le  2r\frac{1}{\sqrt{1-\eps}}.
\]
Therefore, if $\gamma_i([0,1]) \not\subset B^g_{9r}(0)$ for some $i$, we would get
\[
\int_0^1 |\gamma_i'|_{\rm eu} \ge  \frac{1}{\sqrt{1+\eps}} \int_0^1 |\gamma_i'|_g \ge \frac{1}{\sqrt{1+\eps}} 16r,
\]
leading to a contradiction for large $i$, provided $\eps$ is small enough. Hence we can pass to the limit $\gamma_i$ to a curve $\gamma:[0,1] \to \overline{B}^g_{9r}(0)$ such that $\tilde\dist(p,q) = \int_0^1 |\gamma'|_{\rm eu}$. 
Hence $\gamma$ is a critical point for the length functional on $(B^g_{10r}(0), g_{\rm eu})$. Then $\gamma$ is a straight segment from $p$ to $q$ and in particular $\tilde\dist(p,q)= \dist_{\rm eu}(p,q)$.
\end{proof}

\subsection{BV functions and sets of finite perimeter}\label{sec:Perimeter}

Let $(M,g)$ be a complete $C^0$-Riemannian manifold of dimension $n$. Then we can consider the complete and separable metric measure space $(M,\dist,\mathcal{H}^n)$.
Following \cite{MIRANDA2003, ADM14}, we define the \emph{total variation} \(|Df|(B)\in[0,+\infty]\) of a function \(f\in L^1_{\mathrm{loc}}(M)\) in a Borel set \(B\subseteq M\) as
\begin{equation}\label{eqn:VariazioneTotale}
|Df|(B)\coloneqq\inf_{B\subseteq\Omega\text{ open}}\inf\bigg\{\liminf_{n\to\infty}\int_\Omega |\nabla f_n|\;\bigg|\;
(f_n)_{n\in\N}\subset\mathrm{Lip}_{\mathrm{loc}}(\Omega),\,f_n\to f\text{ in }L^1_{\mathrm{loc}}(\Omega)\bigg\}.
\end{equation}

\begin{remark}\label{remlip}
On a $C^0$-Riemannian manifold, the previous definition of total variation is equivalent to the usual one adopted on metric measure spaces. More precisely, defining the slope of a locally Lipschitz function $f$ by
\begin{equation}\label{eqn:lip0}
\mathrm{lip}f(x):=\limsup_{y\to x}\frac{|f(x)-f(y)|}{\dist(y,x)},
\end{equation}
then plugging $\lip f_n$ in place of $|\nabla f_n|$ in \eqref{eqn:VariazioneTotale} yields the same quantity. Indeed, if $f \in \Lip_{\rm loc}(M)$, then $\lip f = |\nabla f|$ almost everywhere, see \cref{lem:CoareaContinuousManifold}.
\end{remark}

If for some open cover \((\Omega_n)_{n\in\N}\) of \(M\) we have that \(|Df|(\Omega_n)<+\infty\) holds for every \(n\in\N\), then \(|Df|\) is a \textit{locally finite Borel measure on \(M\)}.
We say that a Borel set \(E\subseteq M\) is \emph{of locally finite perimeter} if \(P(E,\cdot)\coloneqq|D\chi_E|\) is a locally finite measure, called the \emph{perimeter measure}
of \(E\). When \(P(E)\coloneqq P(E,M)<+\infty\), we say that \(E\) is \emph{of finite perimeter}. If $f\in L^1(M)$ satisfies $|Df|(M)<+\infty$, then we say that $f\in BV(M)$.

Given any \(f\in\mathrm{Lip}_{\mathrm{loc}}(M)\), it holds that \(|Df|\) is a locally finite measure and \(| Df|=\lip f\mathcal{H}^n = |\nabla f| \haus^n\), see \cref{lem:CoareaContinuousManifold}. Moreover, we recall the coarea formula in our setting,
\begin{theorem}[Coarea formula {\cite[Proposition 4.2]{MIRANDA2003}}]\label{thm:Coarea}
Let $(M,g)$ be a complete $C^0$-Riemannian manifold. Let \(f\in L^1_{\mathrm{loc}}(M)\) be such that
\(|Df|\) is a locally finite measure. Fix a Borel set \(E\subseteq M\). Then
\(\R\ni t\mapsto P(\{f<t\},E)\in[0,+\infty]\) is a Borel measurable function and it holds that
\[
|Df|(E)=\int_\R P(\{f<t\},E)\de t.
\]
\end{theorem}

Let us finally introduce the notion of isoperimetric profile.
\begin{definition}
    Let $(M,g)$ be a complete $C^0$-Riemannian manifold with $|M|=+\infty$. The \textit{isoperimetric profile} function is the function $I:(0,+\infty)\to [0,+\infty]$ defined as follows
    \[
    I(V):=\inf\{P(E):E\subset M,\mathcal{H}^n(E)=V\}.
    \]
\end{definition}

\subsection{Technical Lemmas on $C^0$-Riemannian manifolds}

In this section we prove several technical lemmas about $C^0$-Riemannian manifolds we are going to use throughout the paper. Notably, we prove: a precompactness theorem for $BV$ functions on converging sequences of continuous Riemannian manifolds (\cref{lem:PrecompactnessBV}); the fact that every continuous Riemannian manifold is PI on every ball (\cref{lem:TechnicalLemmaC0Manifold}); continuity of the isoperimetric profile of continuous Riemannian manifolds with $C^0$-controlled geometry at infinity (\cref{cor:Continuity}); and a generalized existence theorem for the isoperimetric problem on continuous Riemannian manifolds with $C^0$-controlled geometry at infinity (\cref{thm:MassDecompositionC0}).
\begin{lemma}\label{lem:ApprossimazioneC0Bilipschitz}
    Let $(M,g)$ be a $C^0$-Riemannian manifold. For any $p \in M$, $r>0$ and $\delta>0$ there exists a Riemannian manifold $(N,g^\delta)$ with smooth metric $g^\delta$ such that the inclusion $\iota: (B_r(p), \dist) \to (N, \dist^\delta)$ is well defined and it is $(1+\delta)$-biLipschitz with its image, where $\dist, \dist^\delta$ denote Riemannian distance on $(M,g)$, $(N,g^\delta)$ respectively.
\end{lemma}

\begin{proof}
    For any $\eps\in(0,1)$ let $g^\eps$ be a smooth Riemannian metric on $\overline{B}_{20r}^g(p)$ such that $|g(v,v)-g^\eps(v,v)| \le \eps g(v,v)$ for any $v \in T_xM$ and $x \in B_{20r}(p)$. Gluing the boundary of a smooth connected open domain $D$ such that $\overline{B}_{10r}^g(p) \subset D \subset\subset B_{20r}^g(p)$ with a half-cylinder $[0,+\infty)\times \partial D$ and suitably extending the metric $g^\eps$, we obtain a smooth complete Riemannian manifold $(N,g^\eps)$. Denote by $\dist^\eps$ the Riemannian distance on $(N,g^\eps)$ and by $\tilde\dist$ the distance defined in \cref{lem:MetricheVicineImplicaBilipschitz} with $F=\iota:B_{10r}^g(p) \to B_{10r}^g(p)\subset (N,g^\eps)$. Arguing as in the second part of the proof of \cref{lem:LocalePiattezzaC0metrics}, it follows that $\tilde\dist=\dist^\eps$ on $B_r^g(p)$. Indeed, as in the proof of \cref{lem:LocalePiattezzaC0metrics} we find that for any $x,y \in B^g_r(p)$ there exists a constant speed curve $\gamma:[0,1]\to B^g_{10r}(p)$ such that $\tilde\dist(x,y)= \int_0^1 |\gamma'|_{g^\eps}$, hence $\gamma$ is a geodesic for the metric $g^\eps$. For every $x \in B^g_r(p)$, we have that for almost every $y \in B^g_r(p)$ there exists a unique, hence minimizing, geodesic in $(N,g^\eps)$ from $x$ to $y$. Then, for every $x \in B^g_r(p)$, we have that for almost every $y \in B^g_r(p)$ the curve $\gamma$ just obtained must be the minimizing geodesic from $x$ to $y$ in $(N,g^\eps)$. Hence, given $x \in B^g_r(p)$, there holds $\tilde\dist(x,y)=\dist^\eps(x,y)$ for almost every $y \in B^g_r(p)$. Thus by continuity $\tilde\dist=\dist^\eps$ on $ B^g_r(p)$.    
    Therefore choosing $\eps$ small enough and eventually renaming $g^\eps$ into $g^\delta$, the conclusion follows from \cref{lem:MetricheVicineImplicaBilipschitz}.
\end{proof}

\begin{remark}[Representation of the perimeter]\label{rem:RepresentationPerimeter}
    Let $(M,g)$ be an $n$-dimensional $C^0$-Riemannian manifold, and let $E$ be a set of finite perimeter. Then $P(E,\cdot)=\mathcal{H}^{n-1}\llcorner \partial^eE$, being  $\partial^eE:=M\setminus (E^{(0)}\cup E^{(1)})$ the so-called \emph{essential boundary}, where
\begin{equation}\label{eqn:E1E0}
E^{(1)}\coloneqq\bigg\{x\in M\;\bigg|\;\lim_{r\to 0}\frac{|E\cap B_r(x)|}{|B_r(x)|}=1\bigg\},\qquad
E^{(0)}\coloneqq\bigg\{x\in M\;\bigg|\;\lim_{r\to 0}\frac{|E\cap B_r(x)|}{|B_r(x)|}=0\bigg\}.
\end{equation} 
Indeed, fix on a compact ball $B\subset M$ a sequence $g_i$ of smooth metrics such that $g_i\to g$ uniformly on $B$ (see, e.g., \cref{lem:ApprossimazioneC0Bilipschitz}). Then, up to subsequences, $\partial^e E\cap B$ does not depend on the metrics $g_i,g$ chosen in the definition \eqref{eqn:E1E0} (compare also with item (1) in \cref{lem:TechnicalLemmaC0Manifold}). Moreover, by the very definition of Hausdorff measure, $(1-o_i(1))\mathcal{H}^{n-1}_{g_i}\llcorner B \leq \mathcal{H}^{n-1}_{g}\llcorner B \leq (1+o_i(1))\mathcal{H}^{n-1}_{g_i}\llcorner B$. From the classical De Giorgi-Federer's theorem, $|D\chi_E|_i(\Omega)=\mathcal{H}^{n-1}_{g_i}(\partial^eE\cap \Omega)$ for every $\Omega\subset\subset B$. Thus, taking into account \cref{lem:CoareaContinuousManifold} one has that $|D\chi_E|_i(\Omega)\to |D\chi_E|(\Omega)$ for every Borel $\Omega\subset\subset B$, and thus we conclude $|D\chi_E|\llcorner B=\mathcal{H}^{n-1}\llcorner (\partial^e E\cap B)$. Since $B$ was arbitrary, we get the sought claim. In the following, if $E$ is a set of locally finite perimeter, integrals over its essential boundary are tacitly understood to be taken with respect to the perimeter measure.
\end{remark}

\begin{lemma}\label{lem:PrecompactnessBV}
    Let $(M_i,g_i,p_i)$ be a sequence of pointed $C^0$-Riemannian manifolds of dimension $n$ converging in $C^0$-sense to a pointed $C^0$-manifold $(M,g,p)$. Denote by $\dist, \dist_i$ the Riemannian distances on $(M,g), (M_i,g_i)$ respectively. Let $f_i \in BV(M_i)$ be such that $\sup_i \| f_i\|_{L^1(M_i)} + |D f_i|(M_i) <+\infty$.
    Then, up to subsequence, there exist $f \in BV(M)$, $R_i\nearrow+\infty$ and $(1+1/i)$-biLipschitz embeddings $F_i: (B_{R_i}(p), \dist)\to (M_i,\dist_i)$ with $F_i(p)=p_i$ such that the functions $f_i\circ F_i$ converge to $f$ in $L^1_{\rm loc}(M)$. Moreover there holds
    \[
    |D f|(M) \le \liminf_i |D f_i|(M_i).
    \]
    If also ${\rm spt} f_i \subset B_R(p_i)$ for some $R>0$ and for every $i$, then the convergence occurs in $L^1(M)$.
\end{lemma}

\begin{proof}
By a diagonal argument, up to passing to a subsequence, by \cref{rem:BiLipschitz} there exist  $R_i\nearrow+\infty$ and $(1+1/i)$-biLipschitz embeddings $F_i: B_{R_i}(p)\to M_i$ with $F_i(p)=p_i$. Denote $h_i\eqdef f_i\circ F_i$. For any $r>0$, for $i$ large we have that $h_i \in BV(B_r(p))$ with $\sup_i \| h_i\|_{L^1(B_r(p))} + |D h_i|(B_r(p)) <+\infty$. If we show that $h_i$ admits a subsequence converging in $L^1(B_{r/2}(p))$, the first part of the statement follows. Indeed, by \cref{lem:ApprossimazioneC0Bilipschitz} we can find a smooth Riemannian manifold $(N,g^\delta)$ such that it is well defined the inclusion $\iota:(B_r(p),\dist) \to (N,\dist^\delta)$ and $\iota$ is $2$-biLipschitz with its image. Hence $h_i$ can be seen as an equibounded sequence in $BV(B_r^g(p), g^\delta)$. By classical precompactness we can extract a subsequence converging in $L^1(B_{r/2}^g(p), \haus^n_{g^\delta})$. Since $\haus^n_{g^\delta}$ and $\haus^n_g$ are equivalent, the subsequence converges in $L^1(B_{r/2}(p))$ as well.

The lower semicontinuity inequality readily follows since, for any $r>0$ for $i$ large enough we have
\[
|D h_i|(B_r^g(p)) \le (1+o(1))|D f_i| (B_{2(1+1/i)r}^{g_i}(p_i)) \le (1+o(1))|D f_i| (M_i),
\]
where $o(1)\to0$ as $i\to\infty$.    
\end{proof}

\begin{lemma}\label{lem:TechnicalLemmaC0Manifold}
    Let $(M,g)$ be a $C^0$-Riemannian manifold of dimension $n$. Fix $R>0$, $p\in M$. Then there exists $C:=C(p,R)>1$ such that the following hold.
    \begin{enumerate}
        \item For any $x\in B_R(p)$ and $0<r\leq R$ there holds
        \[
        \frac{|B_{2r}(x)|}{|B_r(x)|} \le C,
        \qquad
        C^{-1}r^n \leq |B_r(x)|\leq Cr^n.
        \]
    
        \item For any $x\in B_R(p)$, any $r\leq R$, and any $f\in\mathrm{Lip}_{\rm loc}(M)$ there holds
        \begin{equation}\label{eq:PoincareOnC0metric}
        \fint_{B_r(x)}\left|f-\fint_{B_r(x)}f\right|\leq Cr\fint_{B_{2r}(x)}|\nabla f|.
        \end{equation}
        
        \item For any $E\subset\subset  B_{R}(p)$ there holds
        \begin{equation}\label{eq:isoperimetricLocaleC0metric}
            |E|^{\frac{n-1}{n}}\leq CP(E).
        \end{equation}
    \end{enumerate}
\end{lemma}

\begin{proof}
The claims are well-known to hold true on manifolds with smooth Riemannian metrics as a consequence of the fact that one can always find a lower bound for the Ricci curvature on compact sets. In fact, on a smooth manifold, (1) is a consequence of the Bishop--Gromov comparison theorem \cite[Lemma 7.1.4]{PetersenBook}, (2) follows from \cite{Buser, RajalaPoincare}, and (3) is a consequence of (1) and (2) by \cite[Theorem 9.7]{HajlaszKoskela} and \cite[Theorem 4.3, Remark 4.4]{Ambrosio02}. Therefore claims (1) and (3) in the setting of the statement readily follow from \cref{lem:ApprossimazioneC0Bilipschitz} applied with $r$ sufficiently large.

It remains to prove (2). Let $(N, g^\delta)$ be given by \cref{lem:ApprossimazioneC0Bilipschitz} with $\delta=1/4$ and radius equal to $3R$. Letting now $f \in \Lip_{\rm loc}(M)$, for $x \in B_R(p)$ and $r \le R$, we have $h\eqdef f \circ \iota^{-1} \in \Lip_{\rm loc}(\iota(B_{2r}(p)), \dist^\delta)$, for $\iota$ as in \cref{lem:ApprossimazioneC0Bilipschitz} where $\dist^\delta$ is the distance on $(N, g^\delta)$. Since $|\nabla h(y)|_{g^\delta} \le 2 |\nabla f (y)|_g$ for any $y \in B_{2r}(p)$, we find
\[
\begin{split}
    \int_{B_r^g(x)}\left|f-\fint_{B_r^g(x)}f\de\vol_g\right|\de\vol_g
    & \le 2 \int_{B_r^g(x)}\left|f-\fint_{B_{\frac54 r}^{g^\delta}(x)}h\de\vol_{g^\delta}\right|\de\vol_g \\
    &\le C \int_{B_{\frac54 r}^{g^\delta}(x)}\left|h-\fint_{B_{\frac54 r}^{g^\delta}(x)}h\de\vol_{g^\delta}\right|\de\vol_{g^\delta} \\
    &\le C r \int_{B_{\frac54 r}^{g^\delta}(x)} |\nabla h |_{g^\delta} \de\vol_{g^\delta}\\
    &\le C r \int_{B^g_{2r}(x)}  |\nabla f (y)|_g \de \vol_g,
\end{split}
\]
where in the third inequality we applied a Poincaré inequality as in \eqref{eq:PoincareOnC0metric}, recalling that on smooth manifolds it is possible to take the integral on the right hand side on the ball of the same radius that appears on the left hand side \cite[Corollary 5.3.5]{SaloffCosteBook}.
\end{proof}

What observed so far implies that $C^0$-Riemannian manifolds locally asymptotic at infinity to space forms are PI spaces, see the forthcoming \cref{cor:C0locasymptoticPi}. We recall their definition, leaving the interested reader to the seminal \cite{HajlaszKoskela, Cheeger00} and to the survey \cite{KleinerMackay}.

\begin{definition}[PI space]\label{def:PI}
Let $(\X,\dist,\meas)$ be a complete and separable metric measure space, where $\meas$ is a Radon measure.
 We say that $\meas$ is {\em uniformly locally doubling} if for every $R>0$ there exists $C_D(R)>0$ such that the following holds 
 $$
 \meas(B_{2r}(x))\leq C_D(R)\meas(B_r(x)), \qquad \forall x\in\X\;\forall r\leq R.
 $$

We say that a {\em weak local $(1,1)$-Poincar\'{e} inequality} holds on $(\X,\dist,\meas)$ if there exists $\lambda$ such that for every $R>0$ there exists $C_P(R)$ such that for every $f\in \mathrm{Lip}(X)$,
the following inequality holds:
$$
\fint_{B_r(x)} |f-\overline f(x)|\de\meas \leq C_P(R)r\fint_{B_{\lambda r}(x)} \mathrm{lip}f\de\meas, 
$$
for every $x\in\X$ and $r\leq R$, where $\overline f(x):=\fint_{B_r(x)}f\de\meas$, and 
\begin{equation}\label{eqn:lip}
\mathrm{lip}f(x):=\limsup_{y\to x}\frac{|f(x)-f(y)|}{\dist(y,x)},
\end{equation}
if $x$ is an accumulation point, or $\mathrm{lip}f(x)=0$ if $x$ is not an accumulation point.

We say that $(\X,\dist,\meas)$ is a {\em PI space} when $\meas$ is uniformly locally doubling and a weak local $(1,1)$-Poincar\'{e} inequality holds on $(\X,\dist,\meas)$.
\end{definition}

\begin{corollary}\label{cor:C0locasymptoticPi}
    Let $(M,g)$ be a $C^0$-Riemannian manifold of dimension $n$ that is $C^0_{\rm loc}$-asymptotic to the $n$-dimensional simply connected complete model $\mathbb H^n_K$ of constant sectional curvature $K\le 0$\footnote{$\mathbb H^n_K=\R^n$ if $K=0$, while $\mathbb H^n_K$ is the $n$-dimensional hyperbolic space of constant sectional curvature $K$ if $K<0$.}. Then $(M, \dist, \haus^n)$ is a PI space.
\end{corollary}

\begin{proof}
    Let $R>0$ and fix $o \in M$. By assumptions and recalling \cref{rem:BiLipschitz}, we can fix $\rho>R$ such that for any $x \in M\setminus B_\rho(o)$ there exists a diffeomorphism $F_x:(B_{3R}(x), g) \to (\mathbb H^n_K, g_K)$ that is $2$-biLipschitz with its image, where $g_K$ is the metric on $\mathbb H^n_K$, {and whose image contains $B_R(F_x(x))$}. Since $\mathbb H^n_K$ is PI, arguing as in the proof of \cref{lem:TechnicalLemmaC0Manifold} it follows that there exist $\overline{C}>0$ such that
    \[
        \frac{|B_{2r}(x)|}{|B_r(x)|} \le \overline C,
        \qquad
        \fint_{B_r(x)}\left|f-\fint_{B_r(x)}f\right|\leq \overline C r\fint_{B_{2r}(x)}|\nabla f|,
    \]
    for any $x\in M \setminus B_\rho(o)$, any $r\leq R$, and any $f\in\mathrm{Lip}_{\rm loc}(M)$.
    
    Next apply \cref{lem:TechnicalLemmaC0Manifold} with $p=o$ and $R=2\rho$. Hence for any $x \in M$ we have that either $x \in B_{2\rho}(o)$, or $x\in M \setminus B_\rho(o)$. Hence the fact that $(M,\dist,\haus^n)$ is PI follows putting together the previous inequalities with those given by \cref{lem:TechnicalLemmaC0Manifold}.
\end{proof}

The next corollary states the local H\"{o}lder continuity of the isoperimetric profile of $C^0$-Riemannian manifolds that are $C^0$-locally asymptotic to a model of constant curvature. The proof essentially follows a classical path, see e.g. \cite[Lemma 2.23]{ANP}. However in this context we do not have an explicit asymptotic rate for the perimeter of balls of infinitesimal radii, which are often used to perturb competitors. In place of ball, we shall employ images of Euclidean balls through biLipschitz maps into the manifold, so to get a one-parameter increasing family of sets whose perimeter has an explicit rate as the parameter goes to zero.

\begin{corollary}\label{cor:Continuity}
Let $(M,g)$ be a $C^0$-Riemannian manifold of dimension $n$ that is $C^0_{\rm loc}$-asymptotic to the $n$-dimensional simply connected complete model $\mathbb H^n_K$ of constant sectional curvature $K\le 0$. Denote by $I$ (resp., $I_K$) the isoperimetric profile of $M$ (resp., $\mathbb{H}^n_K$).

Then $I\le I_K$, and $I$ is locally $\tfrac{n-1}{n}$-H\"{o}lder continuous on $(0,+\infty)$.
\end{corollary}

\begin{proof}
We start by proving that $I\le I_K$. Let $B_t^K(0)\subset \mathbb{H}^n_K$ be a ball with volume $V\in(0,+\infty)$ in $\mathbb{H}^n_K$. If $q_i\in M$ is a diverging sequence of points, for large $i$, up to subsequence, there exist diffeomorphisms $F:(B_{2t}(0), g_K) \to F(B_{2t}(0)) \subset (M,g)$ such that $F(0)=q_i$ and $F$ is $(1+1/i)$-biLipschitz. Hence, there exists $t_i$ such that $t_i\to t$, $E_i\eqdef F(B_{t_i}(0))$ has volume $V$ in $M$, and $P(E_i) \to P(B_t^K(0)) = I_K(V)$. Therefore $I(V) \le \lim_i P(E_i) = I_K(V)$.

Fix $o \in M$. Since $M$ is $C^0_{\rm loc}$-asymptotic to $\mathbb{H}^n_K$, there exists $r>0$ such that for any $p \in M \setminus B_r(o)$ there exists a $2$-biLipschitz map $F_p$ from a Euclidean ball of sufficiently small radius to $M$ with $F_p(0)=p$. Combining with \cref{lem:ApprossimazioneC0Bilipschitz}, we conclude that there exists $\overline{r} \in (0,1)$ such that for any $p \in M$ there exists a $2$-biLipschitz map $F_p:(B_{\overline{r}}^{\rm eu}(0), \dist_{\rm eu} )\to F_p(B_{\overline{r}}^{\rm eu}(0)) \subset M$ such that $F_p(0)=p$. Hence we define the one-parameter family of sets $E_t(p) \eqdef F_p(B_t^{\rm eu}(0))$, for any $p \in M$ and $t \in (0,\overline{r}]$. In particular, recalling also the representation of the perimeter \cref{rem:RepresentationPerimeter} and the fact that the essential boundary is biLipschitz invariant,
\begin{equation}
	P(E_t(p)) \le \overline{C} t^{n-1}, 
	\qquad
	\frac{1}{\overline{C}} t^n \le |E_t(p)| \le \overline{C}t^n,
\end{equation}
for any $p \in M$ and $t \in (0,\overline{r}]$, for some $\overline{C}$ independent of $p,t$.

Combining again the fact that $M$ is $C^0_{\rm loc}$-asymptotic to $\mathbb{H}^n_K$ with item (1) in \cref{lem:TechnicalLemmaC0Manifold} we get that for any $R>0$ there exists a constant $C_R>0$ such that $C_R^{-1}r^n \le |B_r(x)| \le C_R r^n$ for any $x \in M$ and $r\in (0,R)$. Hence recalling that $(M, \dist, \haus^n)$ is also PI by \cref{cor:C0locasymptoticPi}, then it is well-known that there holds a relative isoperimetric inequality in balls of $M$, see \cite[Theorem 5.1]{HajlaszKoskela} and \cite[Remark 4.4]{Ambrosio02}. Since also $\inf_{x\in M} |B_1(x)| >0$ thanks to the asymptotic assumption, it is readily checked that the proof of \cite[Lemma 2.10]{ANP} can be repeated in our setting, yielding that: for any $R>0$ there exists $\widetilde{C}=\widetilde{C}(R)>0$ such that for any $E \subset M$ with $|E| \in (0,+\infty)$ there exists $x_E$ such that
\begin{equation}\label{eq:LocalMassLowerBound}
|E \cap B_r(x_E)| \ge \widetilde{C} \min\left\{ \frac{|E|^n}{P(E)^n} , r^n \right\},	
\end{equation}
for any $r \in (0,R]$.

Fix now $\overline{V} \in (0,+\infty)$ and $\eta \in (0,\overline{V})$. Let also $V_0 \in (\max\{\overline{V}-1, \eta\}, \overline{V}+1)$ and for any $\eps>0$ let $E\subset M$ be a bounded set such that $|E|=V_0$ and $P(E) \le I(V_0) + \eps$ (it is readily checked that, arguing as in \cite[Lemma 2.17]{AFP21}, the isoperimetric profile is achieved by bounded sets).

Let
\[
v\eqdef \min\left\{\min_{x\in M} | E_{\overline{r}}(x)| , 1\right\} \ge \min\{ \overline{C}^{-1}\overline{r}^n, 1\}>0.
\]
Since $E$ is bounded, for any $V \in [V_0,V_0+v)$ there exist $x_V \in M$ and $r_V \le \overline{r}$ such that $|E \cup E_{r_V}(x_V)|=V$ and $|E_{r_V}(x_V)| = V-V_0$. Therefore
\begin{equation*}
	I(V) \le P(E) + P(E_{r_V}(x_V)) \le I(V_0) + \eps + \overline{C} r_V^{n-1}.
\end{equation*}
Since $r_V^n \le \overline{C}(V-V_0)$, letting $\eps\to0$ we conclude that
\begin{equation}\label{eq:Holderianita1}
	I(V) \le  I(V_0) + \overline{C}^{2-\frac1n} (V-V_0)^{\frac{n-1}{n}}.
\end{equation}

We next consider volumes smaller than $V_0$. By \eqref{eq:LocalMassLowerBound} we know that there exists a constant $\widetilde{C}=\widetilde{C}(\overline{r})$ independent of $E$ and a point $x_E$ such that
\[
|E \cap E_t(x_E)| \ge |E \cap B_{\frac{t}{2}}(x_E)| \ge \frac{\widetilde{C}}{2^n} \min\left\{ \frac{|E|^n}{P(E)^n} , t^n \right\} \ge \frac{\widetilde{C}}{2^n} \min\left\{ \frac{V_0^n}{(I_K(V_0)+\eps)^n} , t^n \right\} =  \frac{\widetilde{C}}{2^n} t^n,
\]
for any $t\le t_0$ for some $t_0=t_0(\overline{V}, \eta, n, K, \overline{r}) \in(0, \overline{r})$. Let $\widetilde{C}_2 \eqdef \widetilde{C} 2^{-n} t_0^n>0$. Up to decrease $t_0$, we can assume that $\widetilde{C}_2< \overline{V}$.

If $V_0>\overline{V}-\widetilde{C}_2$, let $V \in ( \max\{\overline{V}-1, \eta, \overline{V}-\widetilde{C}_2\}, \overline{V}+1)$ such that $V < V_0$. Hence there exists $\overline{t} \in (0,t_0)$ such that $|E\setminus E_{\overline{t}}(x_E)| =V$. Similarly as before, we have
\begin{equation*}
	I(V) \le P(E) + P(E_{\overline{t}}(x_E)) \le I(V_0) + \eps +\overline{C} \overline{t}^{n-1}
	\le I(V_0) + \eps + \frac{2^{n-1}\overline{C}}{\widetilde{C}^{\frac{n-1}{n}}}(V_0-V)^{\frac{n-1}{n}},
\end{equation*}
which letting $\eps\to0$ yields
\begin{equation}\label{eq:Holderianita2}
	I(V) \le I(V_0)  + \frac{2^{n-1}\overline{C}}{\widetilde{C}^{\frac{n-1}{n}}}(V_0-V)^{\frac{n-1}{n}}.
\end{equation}

Putting together \eqref{eq:Holderianita1} and \eqref{eq:Holderianita2}, we have proved that there exist $v, C_H(M)>0$ and $\widetilde{C}_2=\widetilde{C}_2(M,\overline{V},\eta)$ such that for any $V_0 \in  ( \max\{\overline{V}-1, \eta, \overline{V}-\widetilde{C}_2\}, \overline{V}+1)$, for any $V \in  ( \max\{\overline{V}-1, \eta, \overline{V}-\widetilde{C}_2\}, \min\{\overline{V}+1, V_0+v\})$ there holds
\begin{equation}\label{eq:Holderianita3}
I(V) \le I(V_0) + C_H|V-V_0|^{\frac{n-1}{n}}.	
\end{equation}
The dependence of the previous constants imply that there exists a neighborhood $U$ of $\overline{V}$ such that \eqref{eq:Holderianita3} holds for any choices of $V,V_0 \in U$. This implies the desired local H\"{o}lder continuity.
\end{proof}

The next theorem is based on a concentration-compactness argument that has been used several times in the literature, applied to the study of the isoperimetric problem in noncompact manifolds. In the smooth setting it has been first obtained in \cite[Theorem 2]{Nardulli}. See also \cite[Theorem 4.48]{RitoreBook} and references therein for a complete account.  Results analogous to \cref{thm:MassDecompositionC0} have been worked out also in the setting of nonsmooth spaces with bounds below on the curvature, see \cite[Theorem 4.6]{AFP21} and  \cite[Theorem 3.3 \& Theorem 1.1]{ANP}.

\begin{theorem}[Asymptotic mass decomposition under $C^0_{\rm loc}$-asymptotic assumptions]\label{thm:MassDecompositionC0}
Let $(M, g)$ be an $n$-dimensional complete $C^0$-Riemannian manifold and assume that $M$ is $C^0_{\rm loc}$-asymptotic to the $n$-dimensional simply connected complete model $\mathbb H^n_K$ of constant sectional curvature $K\le 0$. Fix $o \in M$. Let $V>0$ and let $E_i\subset M$ be a sequence of bounded sets such that $|E_i|=V$ for any $i$ and $\lim_i P(E_i)= I(V)$.

Then, up to subsequence, one of the following alternatives holds true.
\begin{itemize}
    \item The sequence $E_i$ converges in $L^1(M)$ to an isoperimetric set $E$ of volume $V$.

    \item There exist two sequences of radii $R_i, r_i\nearrow+\infty$ with $R_i<r_i$ and a diverging sequence of points $p_i \in M\setminus B_{R_i}(o)$ such that $E_i^c\eqdef E_i \cap B_{R_i}(o)$ converges in $L^1(M)$ to a (possibly empty) isoperimetric set $E$, and $E_i^d\eqdef E_i \cap B_{r_i}(p_i) \setminus B_{R_i}(o)$ converges
    to a ball $B\subset \mathbb H^n_K$ in the sense that
    \[
    \lim_i P(E_i^d) = P(B) , 
        \qquad
    \lim_i |E_i^d| = |B|.
    \]
    Moreover
    \begin{equation*}
        V= |E| + |B|,
        \qquad
        I(V) = P(E) + P(B) .
    \end{equation*}
\end{itemize}
\end{theorem}

\begin{proof}
{
Since the proof of \cref{thm:MassDecompositionC0} is standard and follows closely the strategy of \cite{AFP21, ANP} we just sketch it. 

At first, one gets the analogue of Ritoré--Rosales' result \cite[Theorem 3.3]{ANP} in the setting of \cref{thm:MassDecompositionC0}. Indeed, the proof of \cite[Theorem 3.3]{ANP} uses the coarea formula, the precompactness of $BV$ in $L^1_{\mathrm{loc}}$, and the existence, around every point $p\in M$, of a one-parameter family $\{\mathcal{B}_{p,r}\}_{r\in (0,\varepsilon)}$ of sets such that $r\mapsto |\mathcal{B}_{p,r}|$ is continuous and vanishing as $r\to 0$, and $P(\mathcal{B}_{p,r})\to 0$ as $r\to 0$. The first two come from \cref{lem:PrecompactnessBV}, and \cref{thm:Coarea}, while for the last one it suffices to take pre-images of small Euclidean balls under the map in the second item of \cref{lem:LocalePiattezzaC0metrics}, as it has been done in the proof of \cref{cor:Continuity}. Once this is done, one follows verbatim the proof of \cite[Theorem 4.6]{AFP21}, which additionally needs that $(M,\dist,\mathcal{H}^n)$ is PI, and $|B_r(p)|\geq \zeta r^3$ for every $r\in (0,R]$, and every $p\in M$, where $\zeta,R>0$ are constants depending on $M$: these two properties come from \cref{cor:C0locasymptoticPi} and from its proof (compare also with item (1) in \cref{lem:TechnicalLemmaC0Manifold}). 

By \cref{cor:C0locasymptoticPi}, the metric measure space $(M, \dist, \haus^n)$ is PI, hence it satisfies a relative local isoperimetric inequality (as already explained in the proof of \cref{cor:Continuity}). We can thus directly
follow \cite[Theorem 4.6]{AFP21} until (4.20) there, and then jump to Step 5 there. One gets the following. There is $\overline{N}\in \N  \cup\{+\infty\}$ and radii $R_i\to +\infty$, $T_{i,j}\to_i+\infty$ for  $1\le j< \overline{N}+1$, and there are mutually (with respect to $j$) diverging $p_{i,j}\in M\setminus B_{R_i}(o)$ such that 
\[
E_i^c \xrightarrow[i]{} E \text{ in $L^1(M)$,}
\qquad
(E_i\setminus B_{R_i}(o)) \cap B_{T_{i,j}}(p_{i,j}) \xrightarrow[i]{} B_j \text{ in $L^1(M)$,}
\]
where $B_j$ is a ball in $\mathbb H^n_K$ for any $j< \overline{N}+1$, and $E$ is an isoperimetric set in $M$. The convergence in $L^1$ to the $B_j$'s has to be intended as in the statement of \cref{lem:PrecompactnessBV}, after the composition with biLipschitz embeddings. Moreover one has 
\[
|E|+\sum_{j=1}^{\overline{N}}|B_j| = V, \qquad P(E) + \sum_{j=1}^{\overline{N}} I_K (|B_j|) =
    P(E) + \sum_{j=1}^{\overline{N}} P(B_j) = I(V),
\]
where $I_K$ is the isoperimetric profile of $\mathbb H^n_K$.
Hence either $\overline{N}=0$ and the first item holds, or $\overline{N}\ge1$. In the latter case, we want to show that $\overline{N}=1$, completing the proof of the second item.} By coarea formula we can fix a sequence $\rho_i\nearrow+\infty$ such that $P(E \cap B_{\rho_i}(o) )\le P(E) + 1/i$. For any $i$ we find balls $B_{s_i}(q_i) \subset M\setminus  B_{\rho_i+1}(o)$ such that $|B_{s_i}(q_i)| = V-|E\cap B_{\rho_i}(o)|$ and such that $B_{s_i}(q_i)$ converges to a ball $B\subset \mathbb H^n_K$ with $\lim_i P(B_{s_i}(q_i)) = P(B)$ and $|B|=V-|E|$. If by contradiction $\overline{N}>1$, since the isoperimetric profile $I_K$ is a strictly subadditive function, we get
\[
\begin{split}
    P(E) + \sum_{j=1}^{\overline{N}} I_K (|B_j|) &=
    P(E) + \sum_{j=1}^{\overline{N}} P(B_j) = I(V) \le \liminf_i P(E\cap B_{\rho_i}(o)) + P(B_{s_i}(q_i)) \\&= P(E) + I_K(|B|) 
    \leq P(E) + I_K\left(\sum_{j=1}^{\overline{N}}|B_j|\right) < P(E) + \sum_{j=1}^{\overline{N}} I_K (|B_j|),
\end{split}
\]
which is a contradiction.
\end{proof}

\section{Local inverse mean curvature flow}\label{sec:IMCFWithGlobalIsoperimetric}

We start by recalling the definition of weak inverse mean curvature flow (IMCF) as introduced in \cite{HuiskenIlmanen}.
\begin{definition}[Weak IMCF - Level set formulation]\label{def:WeakIMCFLevelSet}
    Let $(M,g)$ be a smooth complete Riemannian manifold. Given a precompact  $K\subset M$, a locally Lipschitz function $u:M\to\mathbb R$, and a set of locally finite perimeter $E$, define
    \[
    J_u^K(E):=P(E,K)-\int_{E\cap K}|\nabla u|.
    \]

    Let $\Omega\subset M$ be an open set. A function $u\in\mathrm{Lip}_{\mathrm{loc}}(\Omega)$ is called a \textit{weak solution} (resp., subsolution, supersolution) to the inverse mean curvature flow (IMCF) in $\Omega$ if
\[
J_u^K(\{u<t\})\leq J_u^K(E),
\]
for all $t\in\mathbb R$, all $K\subset\subset \Omega$, and all sets $E$ (resp., $E\supset \{u<t\}$, $E\subset \{u<t\}$) such that $E\Delta \{u<t\}\subset K$.
\end{definition}

\begin{remark}
    By virtue of \cite[Lemma 1.1]{HuiskenIlmanen}, a function $u\in\mathrm{Lip}_{\mathrm{loc}}(\Omega)$ is a weak solution the IMCF in $\Omega$ if and only if
    \begin{equation}\label{eqn:IMCFFunctionFormulation}
    \int_K |\nabla u|+u|\nabla u| \leq \int_K |\nabla v|+v|\nabla u|,
    \end{equation}
    for all $K\subset\subset\Omega$, and all $v\in\mathrm{Lip}_{\mathrm{loc}}(\Omega)$ such that $\{u\neq v\}\subset K$.
\end{remark}
The aim of this section is to show \cref{thm:existenceimcfisoplocalized}, stating that on every punctured ball $B$ centered at $o$ on a smooth complete Riemannian manifold one can define a weak IMCF that is bounded from below explicitly in terms of constants that will nicely behave on metrics $C^0$-close to the flat one. We also gather useful properties of this flow in \cref{rem:PropertiesOfFlow}, and \cref{sec:Connectedness}.
\begin{definition}\label{def:SobolevInequality}
    Let $(M,g)$ be an $n$-dimensional complete $C^0$-Riemannian manifold, and let $1\leq p<n$. Let $\Omega\subset M$ be an open set. We say that \textit{$\Omega$ supports a $(p,p^*)$-Sobolev inequality} if there exists a constant $C>0$ such that
    \begin{equation}\label{eqn:pSobolevInequality}
    \left(\int_M |\psi|^{\frac{np}{n-p}}\right)^{\frac{n-p}{n}} \leq C\int_M |\nabla\psi|^p\,\quad \text{for all $\psi\in\mathrm{Lip}_{\mathrm{c}}(\Omega)$}.
    \end{equation}
    We denote $C_{p,\mathrm{Sob}}(\Omega)$ the smallest constant $C$ for which the latter inequality holds.
\end{definition}

\begin{remark}\label{rem:SobolevToIso}
    It is known that \eqref{eqn:pSobolevInequality} with $p=1$ is equivalent to
    \[
    |E|^{\frac{n-1}{n}}\leq C P(E),\quad \text{for all bounded measurable $E\subset\subset\Omega$}.
    \]
    Indeed, one implication readily comes from the very definition in \eqref{eqn:VariazioneTotale}, while the other comes from an application of the coarea formula in \cref{thm:Coarea}. The latter implication is classical and dates back at least to works of Federer--Fleming and Maz'ya in the 60s, see, e.g., \cite[page 488]{FedererFleming}.
\end{remark}

Let $(M,g)$ be a smooth complete Riemannian manifold. We now provide solutions to the weak IMCF in punctured balls $B_R(o) \setminus \{o\}$. The weak IMCF is obtained in the limit, as $p\to 1^+$, of functions $w_p^R := -(p-1) \log G_p^R$, where $G_p^R$, for $p \in (1, n)$ denotes the $p$-harmonic Green function on $B_R(o)$ with Dirichlet boundary conditions. Namely, $G_p^R$ is the solution to
\begin{equation}
\label{eq:def-green}
\begin{cases}
-\Delta_p G_p^R = \abs{\mathbb{S}^{n-1}}\left(\frac{n-p}{p-1}\right)^{p-1} \delta_o  & \text{on }  B_R(o), \\
\phantom{-\Delta_p} G_p^R = 0  & \text{on }  \partial B_R(o), 
\end{cases}
\end{equation}
where $\delta_o$ is the Dirac delta supported at $o$, and $\abs{\mathbb{S}^{n-1}}$
is the measure of the $(n-1)$-dimensional unit sphere. With the above choice of normalization, it  follows from the blow-up procedure leading to \cite[Theorem 1.1]{kichenassamy-veron} that
\begin{equation}
\label{eq:asygreen}
\begin{split}
\left|\frac{G_p^R(x)}{r(x)^{-{(n-p)}/{(p-1)}}} -1\right| \to 0, 
\qquad
\left|\frac{\abs{\nabla G_p^R(x)}}  {r(x)^{-{(n-1)}/{(p-1)}}} -\frac{n-p}{p-1}\right| \to 0
\end{split}
\end{equation}
as $r(x) \to 0$, where we let $r(x) \eqdef \dist(o, x)$.
A full proof of \eqref{eq:asygreen} can be found in the recent \cite[Theorem 2.1]{bmrsx-preprint}.
We recall that the relative $p$-capacity of a compact $K \Subset B_R(o)$ is defined as
\begin{equation}
    \label{eq:p-cap}
    \mathrm{Cap}_p(K, B_R(o)) := \inf\left\{\int_{B_R(o) \setminus K} \abs{\nabla v}^p \st v\in \mathrm{Lip_c} (B_R(o)), v \geq \chi_K\right\}. 
\end{equation}
The following lemma is well known, and consists essentially in \cite[Lemma 3.8]{holopainen-thesis}. 

\begin{lemma}
\label{lem:expcap}
Let $(M, g)$ be a smooth complete Riemannian manifold, and let $o \in M$, $R > 0$. Let $E_t^R := \{w_p^R \leq t\}$ for $p \in (1, n)$, where $w_p^R \eqdef -(p-1) \log G_p^R$ and $G_p^R$ solves \eqref{eq:def-green}. Then
\begin{equation}
\label{eq:expcap}
    \mathrm{Cap}_p(E_t^R, B_R(o)) = \left(\frac{n-p}{p-1}\right)^{p-1}\abs{\mathbb{S}^{n-1}}e^t ,
\end{equation}
for every $t\in\mathbb R$.
\end{lemma}

\begin{proof}
In this proof, let $G_p := G_p^{R}$, and $w_p := w_p^{R}$ for the ease of notation.
It is well-known that the $p$-capacity defined in \eqref{eq:p-cap} is attained computing the $L^p$-norm of the gradient of the $p$-harmonic function with Dirichlet data equal to $1$ on $\partial K$, and equal to $0$ on $\partial B_R(o)$ (see \cite{heinonen-kilpelainen-martio} for a thorough account on nonlinear potential theory). In particular, one gets, for every $t\in (0,+\infty)$,
\begin{equation}
    \label{eq:capGp}
    \mathrm{Cap}_p(\{G_p \geq t\}, B_R(o)) = \frac{1}{t^p}\int_{\{G_p < t\}} \abs{\nabla G_p}^{p} = \frac{1}{t^p}\int_0^t\int_{\{G_p = s\}} \abs{\nabla G_p}^{p-1} \de \haus^{n-1} \de s.
\end{equation}
On the other hand, a straightforward application of the divergence theorem combined with the $p$-harmonicity of $G_p$ (see e.g. the computations in \cite[Proposition 2.8]{benatti2022minkowski}) yields that $\int_{\{G_p = s\}} \abs{\nabla G_p}^{p-1}$ attains the same value for almost every $s\in (0,+\infty)$. Such constant is computed using \eqref{eq:asygreen} as
\begin{equation*}
    \label{eq:constantconasygreen}
    \int_{\{G_p = s\}} \abs{\nabla G_p}^{p-1}  \de \haus^{n-1}  = \left(\frac{n-p}{p-1}\right)^{p-1}\abs{\mathbb{S}^{n-1}},
\end{equation*}
for almost every $s \in (0, + \infty)$. Plugging it into \eqref{eq:capGp} leaves us with 
\begin{equation}
\label{eq:p-capalmost}
\mathrm{Cap}_p(\{G_p \geq t\}, B_R(o)) = \frac{1}{t^{p-1}} \left(\frac{n-p}{p-1}\right)^{p-1}\abs{\mathbb{S}^{n-1}},
\end{equation}
for any $t \in (0, + \infty)$.
Rewriting it in terms of $w_p$ as stated in \eqref{eq:expcap} completes the proof.
\end{proof}

We denote with $C_P(B_R(o))$ the \textit{Poincar\'e constant of $B_R(o)$}, defined as the smallest constant $C$ such that
\begin{equation}\label{eqn:PoincareConstant}
\fint_{B_r(x)}\left|f-\fint_{B_r(x)}f\right|\leq Cr\fint_{B_{2r}(x)}|\nabla f|
\end{equation}
holds for every $x\in B_R(o)$, $r\leq R$, and $f \in \Lip_{\rm loc}(M)$. We denote with $C_A (B_R(o))$ the \textit{Ahlfors constant  of $B_R(o)$}, defined as the smallest number $C\geq 1$ such that
\begin{equation}\label{eqn:AhlforsConstant}
C^{-1} r^n \leq \abs{B_r(x)} \leq C r^n,
\end{equation}
for every $x \in B_R(o)$, and every $0 <r \leq R$.

Finally, denoting $A_{\rho_1,\rho_2}(o)\eqdef B_{\rho_2}(o)\setminus \overline{B}_{\rho_1}(o)$ for $\rho_2>\rho_1$, we denote
\begin{equation}\label{eqn:CCov}
\begin{split}
C_{\mathrm{cov}} (B_\rho(o))\eqdef \min\Big\{N\in\mathbb N \st 
&\text{$A_{3r/4,5r/4}(o)$ is covered by $N$ open balls of radius $r/2$} \\
&\text{with centers in $A_{3r/4,5r/4}(o)$ for any $0<r\leq \rho$} \Big\}.
\end{split}
\end{equation}

Observe that on any complete smooth Riemannian manifold $(M,g)$ of dimension $n\ge2$, for any $o \in M$ and $R>0$ there exists $\rho\in(0, R/2]$ such that
\begin{equation}\label{eqn:CondizioneConnessione}
    \text{$\forall 0<r\leq \rho,\,\, \forall p,q \in \partial B_{r}(o)\,\, \exists$ continuous curve $\gamma\subset A_{3r/4,5r/4}(o)$ connecting $p$ and $q$.}
\end{equation}

\begin{theorem}
\label{thm:existenceimcfisoplocalized}
Let $(M,g)$ be a complete smooth Riemannian manifold of dimension $n\ge 2$. Fix $o\in M$, and $R>0$, and let $p\in (1,n)$. Let $w_p^{2R} = -(p-1) \log G_p^{2R}$, with $G_p^{2R}$ as in \eqref{eq:def-green}. Let $\rho \in (0,R/2]$ be such that \eqref{eqn:CondizioneConnessione} is satisfied. Then, the following hold.
\begin{enumerate}
\item The sequence of functions $w_p^{2R}$ converges, up to subsequence,
locally uniformly in $B_{2R}(o)\setminus\{o\}$ as $p \to 1^+$ to a weak solution $w$ of the IMCF on $B_{2R}(o) \setminus \{o\}$.
\item The function $w$ satisfies
\begin{equation}\label{eq:imcfproper}
        w(x) \geq (n-1) \log r(x) -  C, \qquad \text{for all $x\in \overline{B}_{R}(o)\setminus\{o\}$},
    \end{equation}
where $C=C\big(n, C_{1, \mathrm{Sob}}(B_{2R}(o)), C_P(B_{2R}(o)), C_A(B_{2R}(o)), R/\rho, C_{\mathrm{cov}}(B_\rho(o))\big)$, and $r(x)\eqdef \dist(o,x)$.
\item It holds $w(x)\to -\infty$ as $x\to o$.
\item For every $r\leq R$, letting $T_r:=(n-1)\log r-C-1$, there holds $\{w\leq T_r\}\subset B_r(o)$.
\end{enumerate}
\end{theorem}

\begin{remark}\label{rem:RisultatiKai}
    If $(M,g)$ in \cref{thm:existenceimcfisoplocalized} supports a $(1,1^*)$-Sobolev inequality, then the existence of a \emph{global} proper weak IMCF issuing from a point follows from a careful modification of the proof of \cite{KaiXu}. We mention that in the very recent \cite{XuPhDThesis}, the existence of a global solution to the weak IMCF together with the quantitative estimate \eqref{eq:imcfproper} has been proven under the sole assumption that the ambient manifold  supports a $(1,1^*)$-Sobolev inequality (see \cite[Theorem B]{XuPhDThesis}).
\end{remark}

\begin{remark}
\cref{thm:existenceimcfisoplocalized} is analogous to the result claimed in \cite[Theorem 1.7]{mari_flowlaplaceapproximationnew_2022}, and it might be argued that \cref{thm:existenceimcfisoplocalized} follows from some adaptation of such a result. However,\cite[Theorem 1.7]{mari_flowlaplaceapproximationnew_2022} relies on the quantitative estimate claimed in \cite[Theorem 3.6]{mari_flowlaplaceapproximationnew_2022}, whose proof seems to contain a gap.
\end{remark}

The proof of \cref{thm:existenceimcfisoplocalized} requires some preliminary steps. The following result is a direct consequence of \cite[Theorem 1.1]{kotschwar_localgradientestimatesharmonic_2009}.

\begin{theorem}\label{thm:boundkotni}
    Let $(M, g)$ be a complete smooth Riemannian manifold, and let $u_p$ be a positive $p$-harmonic function defined in an open set $U \subset M$, for $p \in (1, 2)$. Let $w_p = -(p-1) \log u_p$. Then, for any compact subset $K \subset U$, we have
    \begin{equation}
\label{eq:boundkotni}
        \abs{\nabla w_p} \leq C(K),
    \end{equation}
    where $C(K)$ does not depend on $p \in (1, 2)$.
\end{theorem}
The constant $C(K)$ in \cref{thm:boundkotni} depends on sectional curvature bounds for $g$ on $K$.
The following observation will be useful also for estimating the asymptotic behavior of the Hawking mass at $o$ along the IMCF, see \cref{prop:PerimetroEHawkingsProperty} below.

\begin{remark}
For $p\in (1,2)$, if $u_p$ is a positive $p$-harmonic function on $B_{2R}(o)\setminus\{o\}$, and $\mathrm{Sec}\geq -k^2$ on $B_{2R}(o)$, then there exist two constants $\eta:=\eta(k,n)$ and $\zeta:=\zeta(k,n)$ such that for every $x\in B_{2R}(o)\setminus\{o\}$ with $\dist(x,o)<\min\{\eta,R\}$ there holds
    \begin{equation}\label{eqn:SharpGradientKotNi}
    |\nabla w_p|(x) \leq \frac{\zeta}{\dist(o,x)}. 
    \end{equation}
    Indeed, it suffices to apply \cite[Equation (1.5)]{kotschwar_localgradientestimatesharmonic_2009} on balls $B(x,\dist(o,x)/2)$. We remark that both $\eta,\zeta$ can be chosen to be uniform with respect to $p\to 1^+$.
\end{remark}

The following result yields the uniform lower bound on $w_p^{2R}$ that will result in \eqref{eq:imcfproper}.
Its core is in the Harnack inequality for $p$-harmonic functions that comes with the sharp dependence with respect to $p$. The estimate \eqref{eq:stimaluca} below was suggested to the authors by Luca Benatti.

\begin{theorem}\label{thm:EstimatepHarmonic}
Let $(M,g)$ be a complete smooth Riemannian manifold of dimension $n\ge 2$. Fix $o\in M$, and $R>0$, and let $p \in (1,n)$. Let $w_p^{2R} = -(p-1) \log G_p^{2R}$, with $G_p^{2R}$ as in \eqref{eq:def-green}. Let $\rho \in (0,R/2]$ be such that \eqref{eqn:CondizioneConnessione} holds.
    
Then 
\begin{equation}
\label{eq:estimatepharmonic}
        w_p^{2R}(x) \geq (n-p) \log r(x) -  C, \qquad \text{for all $x\in B_{3R/2}(o)\setminus\{o\}$},
    \end{equation}
for $C=C\big(n, C_{1, \mathrm{Sob}}(B_{2R}(o)), C_P(B_{2R}(o)), C_A(B_{2R}(o)), R/\rho, C_{\mathrm{cov}}(B_\rho(o))\big)$, where $r(x)\eqdef \dist(o,x)$.
\end{theorem}

\begin{proof}
Within this proof let $G_p^{2R} =: G_p$, and $w_p^{2R} =: w_p$ for the ease of notation. Let $m(r)=\max_{\partial B_r(o)} w_p$ for $r \in (0, 3R/2)$. Notice that $w_p(x)\to -\infty$ as $x\to o$. Then, by the maximum principle, $B_r (o) \subset \{w_p \leq m(r)\}$, the monotonicity of the $p$-capacity with respect to inclusion and \eqref{eq:expcap} give
\begin{equation*}
    \label{eq:stimaluca1}
    \mathrm{Cap}_p (\overline{B}_r(o), B_{2R}(o)) \leq \mathrm{Cap}_p (\{w_p \leq m(r)\}, B_{2R}(o)) = \left(\frac{n-p}{p-1}\right)^{p-1}\abs{\mathbb{S}^{n-1}} e^{m(r)}, 
\end{equation*}
that results in
\begin{equation}\label{eq:stimaluca}
    m(r) \geq \log (\mathrm{Cap}_p (\overline{B}_r(o), B_{2R}(o))) - (p-1) \log \left(\frac{n-p}{p-1}\right) - \log \abs{\mathbb{S}^{n-1}}.
\end{equation}
It is now well-known that capacities can be estimated exploiting isoperimetric inequalities. More precisely, setting $C =  1/ C_{1, {\rm Sob}}(B_{2R}(o))$, we can apply \cite[Eq. (1.7) p. 141]{grigoryan-isocap} to get
\begin{equation*}
\begin{split}
\mathrm{Cap}_p (\overline{B}_r(o), B_{2R}(o)) &\geq \left(\int_{\abs{B_r(o)}}^{\abs{B_{2R}(o)}}\frac{1}{C v^{\frac{p(n-1)}{n(p-1)}}}  \de v \right)^{1-p}  \\
&= \left(C \frac{(n-p)}{n(p-1)} \right)^{p-1} \abs{B_{r}(o)}^{\frac{n-p}{n}} \left[1 - \left(\frac{\abs{B_r(o)}}{\abs{B_{2R}(o)}} \right)^{\frac{n-p}{n(p-1)}}\right]^{1-p} \\
&\geq \left(C \frac{(n-p)}{n(p-1)} \right)^{p-1} \abs{B_{r}(o)}^{\frac{n-p}{n}}.
\end{split}
\end{equation*}
We can now estimate $\abs{B_{r}(o)}^{\frac{n-p}{n}} \geq C_A(B_{2R}(o))^{\frac{n-p}{n}} r^{n-p}$, and so we get from \eqref{eq:stimaluca} that 
\begin{equation}\label{eq:stimaluca2}
m(r) \geq (n-p) \log r - C
\end{equation}
where now $C=C\big(n, C_{1, \mathrm{Sob}}(B_{2R}(o)), C_A(B_{2R}(o)) \big)$, and it is independent of $p$ as $p\to 1^+$. 

Let us consider now $y\in \partial B_r(o)$ for $r \in (0,\tfrac32 R)$. We aim at estimating $w_p(y)$ combining \eqref{eq:stimaluca2} with a Harnack inequality.
The Harnack inequality for positive $p$-harmonic functions \cite[Theorem 1.28]{salvatori-rigoli-vignati} applied to $G_p$ reads
\begin{equation}\label{eqn:LaBenedettaHarnack}
G_p(y) \leq C_H^{\frac{1}{p-1}}(B_s(z)) G_p(x),
\end{equation}
for any $x, y \in B_{4s/5}(z)$ such that $B_s(z)\subset\subset B_{2R}(o)\setminus\{o\}$, and the Harnack constant $C_H(B_s(z))$ can be estimated from above in terms of $n, C_P(B_{2R}(o)), C_A(B_{2R}(o))$ and $C_{p, \mathrm{Sob}}(B_{2R}(o))$. The interested reader might consult \cite[Theorem 3.4 and Remark 3.5]{mari_flowlaplaceapproximationnew_2022} for the full computations leading to the explicit constant in the Harnack inequality. 
Since $C_{p,\mathrm{Sob}}(B_{2R}(o))\to C_{1,\mathrm{Sob}}(B_{2R}(o))$ as $p\to 1^+$, then \[
C_H(B_s(z)) \le C=C\big(n, C_P(B_{2R}(o)), C_A(B_{2R}(o)), C_{1, \mathrm{Sob}}(B_{2R}(o)) \big)
\]
for any $B_s(z)\subset\subset B_{2R}(o)\setminus\{o\}$.

Fix $x, y \in \partial B_r(o)$ such that
\[
G_p(x) = \min_{\partial B_r(o)} G_p,
\qquad
G_p(y) = \max_{\partial B_r(o)} G_p.
\]
We first consider the case $r \le\rho$, where $\rho$ is as in the assumptions. Hence by \eqref{eqn:CondizioneConnessione} we can connect $x$ and $y$ with a curve $\gamma\subset A_{3r/4,5r/4}(o)$. Letting $N\eqdef C_{\rm cov}(B_\rho(o))$, we can fix a family of at most $N$ balls $\{ B_{r/2}(z_j)\}$ such that $z_j \in A_{3r/4,5r/4}(o)$ and $A_{3r/4,5r/4}(o) \subset \cup_j  B_{r/2}(z_j)$. The existence of $\gamma$ implies that we can find a sequence of points $y:=x_1,\ldots,x_{N'}=:x$ such that $N' \le N+1$, $x_i,x_{i+1}\in \overline{B}_{r/2}(z_{j_i})$ for every $i=1,\ldots,N'-1$ for some $z_{j_i}$. Thus applying iteratively \eqref{eqn:LaBenedettaHarnack} we deduce (the value of the constant $C$ might change from line to line)
\begin{equation}\label{eq:zzPostHarnackScaleBasse}
    \max_{\partial B_r(o)} G_p \le C^{\frac{N'}{p-1}} \min_{\partial B_r(o)} G_p \le C^{\frac{1}{p-1}} \min_{\partial B_r(o)} G_p,
\end{equation}
for $C=C\big(n, C_P(B_{2R}(o)), C_A(B_{2R}(o)), C_{1, \mathrm{Sob}}(B_{2R}(o)), C_{\rm cov}(B_\rho(o)) \big)$.\\
If instead $r \in (\rho, \tfrac32 R)$, then we consider $p_1, p_2 \in \partial B_\rho(o)$ such that $p_1$ (resp. $p_2$) belongs to the intersection of $\partial B_\rho(o)$ with a minimizing geodesic from $y$ to $o$ (resp. from $o$ to $x$). By \eqref{eq:zzPostHarnackScaleBasse} we already know that
\[
G_p(p_1) \le C^{\frac{1}{p-1}} G_p(p_2).
\]
Applying \eqref{eqn:LaBenedettaHarnack} with $s=\rho/4$ iteratively along the geodesic from $y$ to $o$ we find
\[
G_p(y) \le  C^{\frac{N''}{p-1}}  G_p (p_1),
\]
with $N'' \in \N$ such that $N''\le \tfrac32 R / (\tfrac{\rho}{8}) +1$. Arguing analogously along the geodesic from $o$ to $x$, we finally get that
\begin{equation}\label{eq:zzPostHarnack}
    \max_{\partial B_r(o)} G_p \le C^{\frac{1}{p-1}} \min_{\partial B_r(o)} G_p,
\end{equation}
for $C=C\big( n, C_P(B_{2R}(o)), C_A(B_{2R}(o)), C_{1, \mathrm{Sob}}(B_{2R}(o)), C_{\rm cov}(B_\rho(o)) , R/\rho\big)$.
Rewriting \eqref{eq:zzPostHarnack} in terms of $w_p$ and combining with \eqref{eq:stimaluca2} yields
\[
\min_{\partial B_r(o)} w_p = w_p(y) \ge w_p(x) - C = \max_{\partial B_r(o)} w_p - C
\ge (n-p)\log r - C,
\]
for $C=C\big(n, C_{1, \mathrm{Sob}}(B_{2R}(o)), C_P(B_{2R}(o)), C_A(B_{2R}(o)), R/\rho, C_{\mathrm{cov}}(B_\rho(o))\big)$.
\end{proof}

In order to pass to the limit $w_p = -(p-1) \log G_p$ as $p \to 1^+$ on compact sets $K \subset B_{2R}(o)\setminus\{o\}$, we need also some uniform upper bound on $w_p$. This is a well-known consequence of the Laplacian comparison theorem {and of the asymptotics of the Green function at the pole in \eqref{eq:asygreen}}. A proof of the following result can be found in \cite[Theorem 2]{Kura}.
\begin{proposition}\label{prop:BoundbelowG}
    Let $(M, g)$ be an $n$-dimensional smooth complete Riemannian manifold, and let $p\in (1,n)$. Let $G_p$ be the solution of \eqref{eq:def-green} with $\Omega$ in place of $B_R(o)$. Let $r(x):=\dist(o,x)$. Fix $R>0$, and assume $B_R(o)\subset\Omega$. Suppose that $\mathrm{Ric} \geq - (n-1)a$ holds on $B_R(o)$, for some $a > 0$. 
    
    Then
    \begin{equation}
\label{eq:altrocontrollo}
        G_p(x) \geq \int_{r(x)}^R (v_a(t))^{-\frac{1}{p-1}} \de t, 
    \end{equation}
    for any $x \in B_R(o)\setminus\{o\}$, where $v_a(t):= c_{n,p}\left(\left(\sqrt{a}\right)^{-1}\sinh(\sqrt{a}t)\right)^{n-1}$, for suitable $c_{n,p}>0$.   
    Moreover $c_{n,p} \to c_n>0$ as $p\to1^+$.
\end{proposition}
\begin{remark}
    In \cref{prop:BoundbelowG}, the function 
    \[
    \int_{r(x)}^R (v_a(t))^{-\frac{1}{p-1}} \de t
\]
is the solution of \eqref{eq:def-green} in the space form of constant curvature $a$.
\end{remark}

\begin{proof}[Proof of \cref{thm:existenceimcfisoplocalized}.]
Within this proof let $G_p^{2R} =: G_p$, and $w_p^{2R} =: w_p$ for the ease of notation. Let $K \subset B_{2R}(o) \setminus \{o\}$ be compact. We first show that $w_p$ is equibounded on $K$ as $p \to 1^+$. Indeed by \eqref{eq:estimatepharmonic} we know that $w_p\ge (n-p) \log(r(x)) -C$ on $B_{3R/2}(o)\setminus\{o\}$ for $C$ as in \cref{thm:EstimatepHarmonic}. On the other hand, taking $\overline{r}\in(0,2R)$ such that $K \subset B_{\overline{r}}(o)\setminus\{o\}\subset\subset B_{2R}(o)$, if $\mathrm{Ric} \geq -(n-1)a$ on $B_{2R}(o)$ for some $a >0$, by \eqref{eq:altrocontrollo} we have that for all $x\in B_{\overline{r}}(o)\setminus\{o\}$ there holds
\begin{equation}\label{eq:upperboundw}
w_p(x) \leq - \log\left(\int_{r(x)}^{\overline{r}} (v_a(t))^{-\frac{1}{p-1}} \de t\right)^{(p-1)} \leq - \log\left(\abs{r(x)- \overline{r}}^{p-2} \int_{r(x)}^{\overline{r}} (v_a(t))^{-1}\de t\right).
\end{equation}
Hence the above upper bound is uniform with respect to $p \to 1^+$ for any $x \in K$. Thus $w_p$ is bounded on $K \cap B_{3R/2}(o)$ uniformly with respect to $p\to1^+$. Therefore, applying the gradient bound \eqref{eq:boundkotni} on $K$, we deduce that $w_p$ is bounded on $K$ uniformly with respect to $p\to1$. Exhausting $B_{2R}(o)\setminus\{o\}$ with a sequence of increasing compact sets, by Ascoli--Arzel\`{a} and by a diagonal argument, we get that $w_p$ converges to some function $w$ locally uniformly on $B_{2R}(o)\setminus\{o\}$, up to passing to a subsequence with respect to $p$.

We claim that $w$ satisfies the weak formulation of the IMCF on $B_{2R}(o)\setminus\{o\}$. Indeed, arguing as in \cite[Equation (9)]{MoserIMCF}, one gets that $|\nabla w_p|^p\vol$ converges to $|\nabla w|\vol$ in duality with bounded continuous functions on compact subsets contained in $B_{2R}(o)\setminus \{o\}$, along the suitable sequence $p_i\to 1^+$ for which $w_{p_i}$ converges. Using again \cite[Equation (9)]{MoserIMCF} this is enough to conclude that $w$ is a weak solution of the IMCF on $B_{2R}(o)\setminus\{o\}$. 

The lower bound \eqref{eq:estimatepharmonic} passes to the limit and gives \eqref{eq:imcfproper}. The upper bound \eqref{eq:upperboundw} is preserved on compact subsets of $B_{3R/2}(o) \setminus \{o\}$ for the limit $w$ as well, hence it shows that $w \to -\infty$ as $x  \to o$.

Therefore we proved items (1), (2) and (3) of the statement. 
It remains to show that $\{w \le (n-1)\log r - C-1\} \subset B_r(o)$ for every $r\leq R$. Fix $r\in (0,R]$. By \cref{thm:EstimatepHarmonic} there exists 
\[
C=C\big(n, C_{1, \mathrm{Sob}}(B_{2R}(o)), C_P(B_{2R}(o)), C_A(B_{2R}(o)), R/\rho, C_{\mathrm{cov}}(B_\rho(o))\big),
\] 
such that $w_p \ge (n-p) \log r - C$ on $\partial B_r(o)$. Denoting $T_r \eqdef (n-1) \log r - C -1$, it follows that $w_p \ge T_r + \tfrac12$ on $\partial B_r(o)$ for any $p$ sufficiently close to $1$. By \eqref{eq:imcfproper}, it follows that $\{ w \le T_r\} \cap  \partial B_r(o) = \emptyset$. Suppose by contradiction that there exists $z \in B_{2R}(o) \setminus B_r(o)$ such that $w(z) \le T_r$. Then $w_{p_i}(z) \le T_r+ \tfrac14$ for a sequence $p_i\to 1^+$ such that $w_{p_i}$ converges to $w$, and for $i$ large enough.
Since $w_{p_i} \ge T_r + \tfrac12$ on $\partial B_r(o)$ and $w_{p_i}(x)\to+\infty$ as $r(x)\to 2R$, then $w_{p_i}$ would have an interior minimum on $B_{2R}(o) \setminus \overline{B}_r(o)$, which is a contradiction to the maximum principle for $p$-harmonic functions. 
\end{proof}

\subsection{Properties of the weak IMCF}\label{rem:PropertiesOfFlow}

Let us collect in this section few known properties on the weak IMCF constructed in \cref{thm:existenceimcfisoplocalized}.

Let $(M,g)$ be a smooth complete Riemaniann manifold. Let $o\in M, R>0$ and let $w$ be given by \cref{thm:existenceimcfisoplocalized}. Let $T\in\R$ be such that $\{w \le T\} \subset B_R(o)$ (such a $T$ exists by \cref{thm:existenceimcfisoplocalized}).
For every $t\in (-\infty,T]$, set $E_t:=\{w<t\}$, and $E_t^+ = \{w \le t\}$.
Then the following holds.
    \begin{enumerate}
        \item If $n\le 7$, then for every $t\in (-\infty,T]$ both the set $E_t$ and the set $E_t^+$ have $C^{1,\alpha}$ boundary, for some $\alpha>0$. Indeed, as a direct consequence of \cref{def:WeakIMCFLevelSet}, the previous sets are local $(\Lambda,r_0)$-minimizers in $B_R(o)\setminus\{o\}$, see, e.g., \cite[Example 21.2]{MaggiBook}; thus by (the Riemannian analogue of) \cite[Theorem 21.8]{MaggiBook} one gets the sought claim. In fact, one can obtain $C^{1, 1}$ regularity of the sublevel sets \cite{Heidusch}.\\
        It follows that $\partial E_t = \partial \{ w \le t\}=\{w=t\}$ for almost every $t \le T$, and that $\partial E_s \to \partial E_t$ in $C^1$ as $s\nearrow t$, for every $t \le T$, cf. \cite[Eq. (1.10)]{HuiskenIlmanen}. \\
        Moreover, it is meaningful to speak about the weak mean curvature $H$ on $\partial E_t$, see, e.g., \cite[page 16]{HuiskenIlmanen}. Moreover, $H=|\nabla w|>0$ $\mathcal{H}^{n-1}$-almost everywhere on $\partial E_t$ for almost every $t\in (-\infty,T]$, see \cite[Equation (1.12)]{HuiskenIlmanen}, and \cite[Lemma 5.1]{HuiskenIlmanen}.
        
        \item For almost every $t\in (-\infty,T]$, the boundary $\partial E_t$ has weak second fundamental form $A$ (see \cite[Definition 1.3]{GerochMonotonicityMoser}, or \cite[pages 401--405]{HuiskenIlmanen}), and 
        \[
        \int_{\partial E_t} |A|^2<+\infty.
        \]
        The last inequality is a consequence of \cite[Theorem 1.1(vi)]{GerochMonotonicityMoser}.
        \end{enumerate}
The following Proposition gathers some key properties of the weak IMCF in dimension $3$ constructed through the procedure of the previous section; most notably the Geroch monotonicity formula of Huisken--Ilmanen through such flow.

\begin{proposition}\label{prop:PerimetroEHawkingsProperty}
Let $(M,g)$ be a smooth complete Riemaniann manifold of dimension $n=3$. Let $o\in M, R>0$ and let $w$ be given by \cref{thm:existenceimcfisoplocalized}. Denote $E_t\eqdef \{w <t\}$. Let $T$ be such that $\{w \le T\} \subset B_R(o)$. Let $H$ be the weak mean curvature of the boundary $\partial E_t$. Then the following holds.
            \begin{itemize}
                \item There exist $\bar t\in\mathbb R$, $C_1>0$ such that
                \begin{equation}\label{eqn:AsymptoticsPerimeterAndH}
        {P(E_t)=4\pi e^t}\, \quad \text{for all $t\in (-\infty,T]$}, \quad \text{and} \quad \int_{\partial E_t} H^2 \leq C_1, \quad \text{for almost every $t\in (-\infty,\bar t)$}.
        \end{equation}
        \item For $t\in (-\infty,T]$,  define the \emph{Hawking mass}
\begin{equation}\label{eqn:HawkingMass}
\mathfrak{m}_H(\partial E_t):=\frac{P(E_t)^{1/2}}{(16\pi)^{3/2}}\left(4\pi-\int_{\partial E_t}\frac{H^2}{4}\right).
    \end{equation}
    Then $\mathfrak{m}_H(\partial E_t)\to 0$ for almost every $t\to -\infty$. Moreover, {for a.e. $-\infty<r<s\leq T$, if $\partial E_t$ is connected for a.e. $t \in [r,s]$, then
    \begin{equation}\label{eqn:GerochMonotonicity}
        \begin{split}\mathfrak{m}_H(\partial E_s) &\geq \mathfrak{m}_H(\partial E_r)+ 
         \frac{1}{(16\pi)^{3/2}}\int_r^s P(E_t)^{1/2}\int_{\partial E_t}R_g\de t
        \end{split}.
    \end{equation}
    }
 \end{itemize}
 \end{proposition}

\begin{proof}
Denote $E_t^+\eqdef \{ w\le t\}$ and let $E^p_t\eqdef \{ w^{2R}_p \le t\}$ for $w^{2R}_p$ as in \cref{thm:existenceimcfisoplocalized}. We shall omit the superscript $2R$ in the sequel, as $R$ is fixed. We know from \cref{thm:existenceimcfisoplocalized} that $w_p$ converges to $w$ locally uniformly on $B_{2R}(o)\setminus\{o\}$ along a sequence $p_i\to1$.

We claim that for any $\sigma \in (-\infty,T), \eps>0$ there exists $\tilde p>1$ such that $E^{p_i}_{t-\eps} \subset E_t^+ \subset E^{p_i}_{t+\eps}$ for any $p_i \in (1,\tilde p)$ and $t \in [\sigma, T]$.\\
We prove the containment $E^{p_i}_{t-\eps} \subset E_t^+$ first. Assume by contradiction that there exist $\sigma \in (-\infty, T)$, $\eps>0$ such that, up to subsequence, there exist $t_i \in [\sigma, T]$ such that $E^{p_i}_{t_i-\eps} \not\subset E_{t_i}^+$ for any $i$. Then there exist points $x_i \in B_{2R}(o) \setminus\{o\}$ such that $w_{p_i}(x_i) \le t_i-\eps$ but $w(x_i)>t_i$ for any $i$. Hence $\liminf_i \dist(x_i,o)>0$, for otherwise $-\infty = \liminf_i w(x_i) \ge \sigma$ by \cref{thm:existenceimcfisoplocalized}. Moreover, there exists $\eta>0$ such that $w_{p_i} > T + \eta$ on $\partial B_R(o)$ for large $i$, for otherwise $\{w \le T\} \cap \partial B_R(o)$ would be nonempty. Hence $x_i \in B_R(o)$ for large $i$, for if $x_i \in B_{2R}(o) \setminus \overline{B}_R(o)$, since $w_{p_i}(x_i) \le T-\eps$ and $w_{p_i}(x)\to+\infty$ as $r(x)\to2R$, then $w_{p_i}$ would have an interior minimum on $B_{2R}(o) \setminus\overline{B}_R(o)$, contradicting the maximum principle for $p$-harmonic functions. Hence, up to subsequence, $t_i\to \tau \in [\sigma, T]$ and $x_i\to x \in \overline{B}_R(o) \setminus\{o\}$. But this contradicts the local uniform convergence, as this implies $\tau-\eps \ge \lim_i w_{p_i}(x_i) = w(x)= \lim_i w(x_i) \ge \tau$.\\
The containment $E_t^+ \subset E^{p_i}_{t+\eps}$ follows by an analogous contradiction argument. This time the contradicting sequence of points $x_i \in B_{2R}(o)\setminus\{o\}$ satisfies $w_{p_i}(x_i)>t_i+\eps$ and $w(x_i) \le t_i$, for $t_i \in [\sigma, T]$. Hence $x_i \in \{w\le T\} \subset B_R(o)$. Also $x_i\not\to o$ because of the uniform upper bound \eqref{eq:upperboundw}. Hence we have again the convergence $x_i\to x \in \overline{B}_R(o) \setminus\{o\}$, up to subsequence, and one derives a contradiction as in the previous case.

Recalling \cref{lem:expcap}, for  any $\sigma \in (-\infty,T), \eps>0$ there exists $\tilde p>1$ such that
\[
\begin{split}
4\pi\left( \frac{3-p_i}{p_i-1} \right)^{p_i-1}e^{t-\eps}
&=
\mathrm{Cap}_{p_i}(E^{p_i}_{t-\eps} , B_{2R}(o)) \le \mathrm{Cap}_{p_i}(E_t^+ , B_{2R}(o)) \le \mathrm{Cap}_{p_i}(E^{p_i}_{t+\eps} , B_{2R}(o))
\\&
=
4\pi\left( \frac{3-{p_i}}{{p_i}-1} \right)^{{p_i}-1}e^{t+\eps},
\end{split}
\]
for any $p_i \in (1,\tilde p)$ and $t \in [\sigma, T]$. Letting first $p_i\to 1$ and then $\eps\to0$ and $\sigma\to-\infty$, we find that
\[
\lim_{i} \mathrm{Cap}_{p_i}(E_t^+ , B_{2R}(o)) = 4\pi e^t,
\]
for any $t \in (-\infty, T]$. Then the first equality in \eqref{eqn:AsymptoticsPerimeterAndH} will follow if we prove that
\begin{equation}\label{zzEqConvergenzaCapToP}
    \lim_{p\to 1} \mathrm{Cap}_p(E_t^+ , B_{2R}(o)) = P (E_t),
\end{equation}
for any $t \in (-\infty, T]$.\\
Fix $t \in (-\infty, T]$. It follows from the very definition of weak IMCF as in \cite[Minimizing Hull Property 1.4]{HuiskenIlmanen} that $E_t^+$ is strictly outward minimizing relatively to $B_{2R}(o)$; meaning that whenever $E_t^+ \subsetneq F \subset\subset B_{2R}(o)$, then $P(E_t^+) < P(F)$. It is readily checked that the proof of \cite[Theorem 1.2]{FogagnoloMazzieriHulls} can be localized in the ball $B_{2R}(o)$, thus yielding that
\[
\lim_{p\to 1} \mathrm{Cap}_p(E_t^+ , B_{2R}(o)) = P (E_t^+).
\]
Finally, as in \cite[Minimizing Hull Property 1.4(iv)]{HuiskenIlmanen}, there holds $P(E_t^+)=P(E_t)$, so that \eqref{zzEqConvergenzaCapToP} follows.

We show the second property in \eqref{eqn:AsymptoticsPerimeterAndH}.{ By applying \eqref{eq:upperboundw} with $\overline{r}=2r(x)$ and sending $p=p_i\to 1^+$, we get that there exists $\tilde \eta, \vartheta>0$ such that if $r(x)<\tilde\eta$, then $w(x)\leq 2\log (r(x))+\vartheta$. On the other hand, for $\bar t$ small enough we have $E_t \subset B_{e^{(t+C)/2}}(o)$
        for every $t<\bar t$, see 
        item (4) in \cref{thm:existenceimcfisoplocalized}. Up to possibly taking a smaller $\bar t$, for every $t<\bar t$ we have $\partial E_t \subset\{w=t\}\subset M\setminus B_{e^{(t-\vartheta)/2}}$.
        Thus
        \[
        \esssup_{\partial E_t} |\nabla w| \leq \esssup_{\partial E_t}\frac{\zeta}{\dist(o,x)} \leq \zeta e^{\vartheta/2} e^{-t/2},
        \]
        for almost every $t<\bar t$, where the first inequality comes from the fact that \eqref{eqn:SharpGradientKotNi} passes to the limit as $p_i\to 1^+$. Thus, for almost every $t\in (-\infty,\bar t)$, there holds
        \[
        \int_{\partial E_t} |\nabla w|^2\leq |\partial E_t|\zeta^2e^{\vartheta}e^{-t}=4\pi\zeta^2 e^{\vartheta}=:C_1 < +\infty,
        \]
        and then the first item is proved recalling that $H=|\nabla w|$ $\mathcal{H}^{n-1}$-almost everywhere on $\partial E_t$ for almost every $t\in (-\infty,\bar t)$.

        The fact that $\mathfrak{m}_H(\partial E_t)\to 0$ for almost every $t\to-\infty$ is a direct consequence of the first item and the definition of Hawking mass.\\
        The last part of the second item is the analogue of \cite[Geroch Monotonicity Formula 5.8]{HuiskenIlmanen}, which provides the analogous monotonicity formula for a weak IMCF globally defined on a complete manifold. The validity of \eqref{eqn:GerochMonotonicity} in the present setting is well-known to the experts, and can be deduced as a consequence of \cite[Theorem 4.1]{BenattiPludaPozzetta}. In fact, the estimate \eqref{eqn:GerochMonotonicity}, which does not takes into account all the positive terms present in the complete \cite[Geroch Monotonicity Formula 5.8]{HuiskenIlmanen} (see also \cite[Theorem 5.5]{BenattiPludaPozzetta}), can also be derived by suitably passing to the limit $p\to 1$ in \cite[Theorem 2.11]{BenattiFogagnoloMazzieri23SIGMA}, taking into account \cite[Corollary 4.7, Corollary 4.8]{BenattiPludaPozzetta}.
        }
\end{proof}

\subsection{Connectedness of level sets of the weak IMCF}\label{sec:Connectedness}

In the following lemma we collect some facts about connectedness of level set of the weak IMCF.
This is analogous to \cite[Connectedness Lemma 4.2]{HuiskenIlmanen}, we provide a proof for the convenience of the reader.
    
 \begin{lemma}\label{lem:Connectedness}
         Let $\Omega$ be an open set 
         with Lipschitz boundary 
         in a smooth complete $n$-dimensional Riemannian manifold $(M,g)$ with $n\le 7$, and let $o\in\Omega$. Let $\Omega\subset \subset U$, where $U$ is an open set, and let $w\in\mathrm{Lip}_{\mathrm{loc}}(U\setminus\{o\})$ be a weak solution of the IMCF on $U\setminus\{o\}$, see \cref{def:WeakIMCFLevelSet}. Then the following hold.
         \begin{enumerate}
             \item Let $t\in\mathbb R$. Then every connected component of $\{w<t\}\cap\Omega$ (resp., $\{w>t\}\cap\Omega$) is not relatively compact in $\Omega\setminus\{o\}$.
             \item Assume further that $\partial\Omega$ is connected, and there exist $t_0<t_1$ such that:
             \begin{itemize}
                 \item $\{w<t_1\} \subset\subset \Omega$;
                 \item there exists an open set $\Omega''\ni o$, such that $\Omega''\setminus\{o\}$ is connected, and $\Omega''\setminus\{o\}\subset\{w<t_0\}$.
             \end{itemize}
             Then for every $t\in (t_0,t_1)$ we have that both $\{w<t\}$ and $\{w>t\}\cap\Omega$ are connected.
             \item 
            Let the hypotheses of (2) above be satisfied. Assume further that $\Omega$ is connected, and $H_1(\Omega;\mathbb Z)=\{0\}$. Then $\partial\{w<t\}$ is connected for every $t\in (t_0,t_1)$.
         \end{enumerate}
\end{lemma}

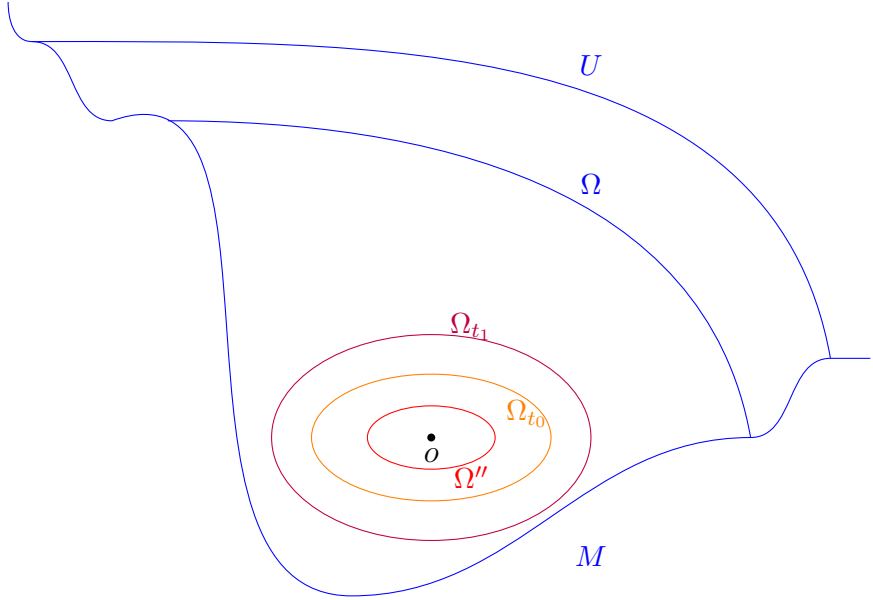
\begin{figure}[h]
    \begin{center}
     \begin{tikzpicture}[scale=1.05]
    \draw[blue] (-1.3,7.5) to[out=270,in=180] (-1,7) to[out=0,in=180] (0,6) to[out=20,in=180] (3,0) to[out=0,in=180] (8,2) to [out=0,in=180] (9,3) to (9.5,3);
    \node[blue] at (6,0.5) {\small$M$};
    
    \draw[blue] (-1,7) to[out=0,in=100] (9,3);
    \node[blue] at (6,6.7) {\small$U$};
    
    \draw[blue] (0.7,6) to[out=0,in=100] (8,2);
    \node[blue] at (6,5.2) {\small$\Omega$};
    
    \fill (4,2) circle (0.05) node[below] {$o$};
    
    
    \draw[purple] (4,2) ellipse (2 and 1.3);
    \node[purple] at (4.5,3.4) {\small$E_{t_1}$};
    
    \draw[orange] (4,2) ellipse (1.5 and 0.8);
    \node[orange] at (5.2,2.3) {\small$E_{t_0}$};
    
    \draw[red] (4,2) ellipse (0.8 and 0.4);
    \node[red] at (4.5,1.5) {\small$\Omega''$};
\end{tikzpicture}
\end{center}
\caption{The picture sketches the assumptions of  \cref{lem:Connectedness}.}\label{fig:DisegnoLemmaTopologico}
\end{figure}

\begin{proof}[Proof of \cref{lem:Connectedness}]
         The proof of item (1) follows verbatim as in \cite[Connectedness Lemma 4.2(i)]{HuiskenIlmanen}. Let us repeat it here for the ease of the reader. Assume by contradiction a connected component $C$ of $\{w>t\}\cap\Omega$ is relatively compact in $\Omega\setminus\{o\}$. Then consider the function $v:=w$ on $U\setminus (C\cup\{o\})$ and $v:=t$ on $C$. By \eqref{eqn:IMCFFunctionFormulation} we get
         \[
         \int_C |\nabla w|+w|\nabla w|\leq \int_C t|\nabla w|.
         \]
         Since $C\subset \{w>t\}$, the latter implies that $|\nabla w|=0$ on $C$, and thus $w=t$ on $C$, which is a contradiction. 
         
         Now assume by contradiction a connected component $C'$ of $\{w<t\}\cap\Omega$ is relatively compact in $\Omega\setminus\{o\}$. Take $\bar t:=\min_{\overline{C'}} w$. Then for $0<\eta<1$ small enough there is a connected component $C''$ of $\{w<\bar t+\eta\}\cap\Omega$ inside $C'$. Notice that $w\geq\bar t>\bar t+\eta -1$ on $C''$. Repeating the previous argument with $\bar t+\eta$ in place of $t$ and $C''$ in place of $C$ gives again a contradiction.

         The item (2) follows from item (1). First, by the assumption in the first bullet, for $t \in (t_0, t_1)$ every connected component of $\{w<t\}$ stays away from $\partial\Omega$. Then, from item (1), for every $t\in (t_0,t_1)$, $o$ is in the closure of any connected component of $\{w<t\}$. Then every connected component of $\{w<t\}$ intersects $\Omega''\setminus\{o\}$. Since $\Omega''\setminus\{o\}$ is connected, and $\Omega''\setminus\{o\}\subset \{w<t_0\}\subset \{w<t\}$, thus every connected component of $\{w<t\}$ contains $\Omega''\setminus\{o\}$. Thus, there exists at most one connected component of $\{w<t\}$, because every connected component contains $\Omega''\setminus\{o\}$, which is itself connected. Similarly, from the hypotheses of the second bullet, for every $t\in (t_0,t_1)$, we have $\partial\Omega\subset\{w>t\}$, and $\{w>t\}\cap\Omega\subset\Omega\setminus\overline{\Omega''}$. Thus, every connected component of $\{w>t\}\cap\Omega$ avoids $o$, and then, from item (1), its closure must then intersect $\partial\Omega$. Since $\partial\Omega$ is connected, there exists at most one connected component of $\{w>t\}\cap\Omega$, arguing as before.

         The item (3) is inspired by \cite[Connectedness Lemma 4.2(ii)]{HuiskenIlmanen}. Analogous arguments have appeared in \cite[Lemma 2.3]{MunteanuWang}, \cite[Lemma 6.1]{JDGLiTam}, and \cite[Lemma 4.46]{LeeBook}. We give here a self-contained proof using item (2) and the Mayer--Vietoris sequence. 

        {Recall from item (1) of \cref{rem:PropertiesOfFlow} that $\partial\{w<t\}$ is $C^{1,\alpha}$ for every $t \in (t_0,t_1)$, and that $\partial \{w<t\} = \{w=t\}$ for almost every $t \in (t_0,t_1)$. It is sufficient to prove that $\partial \{w<t\}$ is connected for $t$ such that $\partial \{w<t\} = \{w=t\}$. Indeed, for $\tau \in (t_0,t_1)$ such that $\partial \{w<\tau\} \neq \{w=\tau\}$, there exists a sequence $t_i\nearrow \tau$ such that $\partial \{w<t_i\} = \{w=t_i\}$. From item (1) of \cref{rem:PropertiesOfFlow} we know that $\partial \{w < t_i\} \to \partial \{ w < \tau\}$ in $C^1$, hence connectedness will be preserved in the limit.
         
        Hence we can assume by contradiction that there exists $t\in (t_0,t_1)$ such that $\partial\{w<t\}=\{w=t\}$ is not connected.
        Since $\partial\{w<t\}$ is $C^{1,\alpha}$, it has a finite number $m\geq 2$ of connected components. Since $\bigcap_{\eta>0} \{t-\eta<w<t+\eta\} = \{w = t\} = \partial \{w < t\}$,} there exists $\eta$ small enough such that $[t-\eta,t+\eta]\subset (t_0,t_1)$ and $\{t-\eta<w<t+\eta\}$ has $m'\geq 2$ connected components. Call $A:=\{w<t+\eta\}\cup\{o\}$ and $B:=\{w>t-\eta\}\cap\Omega$. Notice that $A$ and $B$ are open, and by item (2) they are both connected.
         
         Finally notice that $A\cap B=\{t-\eta<w<t+\eta\}$ is not connected, and $A\cup B=\Omega$. The Mayer--Vietoris exact sequence (where homology is understood with integer coefficients) ends with 
         \[
         \ldots \rightarrow H_1(\Omega) \rightarrow H_0(A\cap B) \rightarrow H_0(A)\oplus H_0(B) \rightarrow H_0(\Omega) \rightarrow 0.
         \]
         Recall that for a topological space $X$ there holds $H_0(X;\mathbb Z)\cong \mathbb Z^\ell$, where $\ell$ is the number of connected components of $X$. Thus by using the assumptions of item (3) the previous exact sequence becomes 
         \[
         \ldots \rightarrow 0 \rightarrow \mathbb Z^{m'}\rightarrow \mathbb Z^2 \rightarrow \mathbb Z \rightarrow 0,
         \]
         from which $\mathbb Z\cong \mathbb Z^2 / \mathbb Z^{m'}$, and thus $m'=1$, which results in a contradiction.
     \end{proof}

\section{Proof of the main results}\label{sec:Proofs}

In this section we prove the main theorems \cref{thm:CorIsop}, \cref{thm:CorMassa}.

\subsection{Producing a set satisfying the reverse Euclidean isoperimetric inequality}\label{sec:Producingset}

In this section we show how, in the hypotheses of \cref{thm:CorMassa} and \cref{thm:small}, we can produce sets  that satisfy the Euclidean reverse isoperimetric inequality with sharp constant.

\begin{lemma}\label{lem:AlmostShi}
    Let $(M,g)$ be a smooth complete Riemannian manifold of dimension $3$. Let $o \in M$ and $R>0$. Let $w$ be given by \cref{thm:existenceimcfisoplocalized}.
    Let $T \in\R$ be such that $\{w<T\}\subset\subset B_\rho(o)$ for some $\rho\le R$. Suppose that $\partial \{w < t\}$ is connected for any $t <T$. If $R_g\geq -\delta$ on $B_\rho(o)$, for some $\delta>0$, then
    \begin{equation}
    \label{eq:quasishi}
    |E_t|\geq \frac{1}{\sqrt{1+\frac{2}{3}\delta e^T}}\frac{P(E_t)^{3/2}}{6\sqrt{\pi}}\qquad \forall t<T,
    \end{equation}
    where $E_t \eqdef \{ w <t\}$.
\end{lemma}

\begin{proof}
We recall that $E_t$ has $C^{1,\alpha}$ boundary, $\partial E_t$ admits weak mean curvature $H$ for all $t<T$, and $H=|\nabla w|>0$ $\haus^2$-a.e. for a.e. $t\in(-\infty,T)$, see item (1) of \cref{rem:PropertiesOfFlow}.
The Hawking mass of a set $\Omega$ with $C^1$ boundary $\partial\Omega$ possessing weak mean curvature $H$ is given by
\[
\mathfrak{m}_H(\partial\Omega)=\frac{P(\Omega)^{1/2}}{(16\pi)^{3/2}}\left(4\pi-\int_{\partial\Omega}\frac{H^2}{4}\right).
\]
As $\{w \le T'\} \subset \subset B_\rho(o)$ for any $T'<T$, we can apply \cref{prop:PerimetroEHawkingsProperty}, which yields that $P( E_t)=4\pi e^t$ for any $t<T$, and that $\mathfrak{m}_H(\partial E_t)\to 0$ for almost every $t\to-\infty$.
Hence the Geroch Monotonicity formula \eqref{eqn:GerochMonotonicity} along a sequence $r\to-\infty$ and $s=t$, together with the fact that $R_g\geq -\delta$ on $B_\rho(o)$, gives that for a.e. $t\in (-\infty, T)$ there holds
    \[
    \mathfrak{m}_H(\partial E_t)\geq -\frac{\delta}{(16\pi)^{3/2}}\int_{-\infty}^t P(E_t)^{3/2}=-\frac{\delta}{(16\pi)^{3/2}}\int_{-\infty}^t(4\pi)^{3/2}e^{3t/2}=-\frac{\delta}{12}e^{3t/2}.
    \]
    Since $\mathfrak{m}_H(\partial E_t) = \frac{2\pi^{1/2}e^{t/2}}{64\pi^{3/2}}\left(4\pi-\int_{\partial E_t}\frac{H^2}{4}\right)$, for a.e. $t\in (-\infty,T)$, we find
\begin{equation}\label{eqn:EstimateOnH2}
     \int_{\partial E_t} H^2\leq 16\pi + \frac{32}{3}\delta\pi e^t.
\end{equation}
    By H\"older inequality we get that for almost every $t\in (-\infty,T)$ there holds
    \begin{equation}\label{eqn:HolderPerimeter}
    P(E_t)\leq \left(\int_{\partial E_t}|\nabla w|^2\right)^{1/3}\left(\int_{\partial E_t}\frac{1}{|\nabla w|}\right)^{2/3}.
    \end{equation}
    Hence, by recalling that $\int_{\partial E_t}|\nabla w|^2=\int_{\partial E_t} H^2$ is finite by \eqref{eqn:EstimateOnH2} for almost all $t\in (-\infty,T)$, recalling that $|\nabla w|>0$ $\haus^2$-a.e. on $\partial E_t$ for a.e. $t \in (-\infty,T)$, and by using the coarea formula together with \eqref{eqn:HolderPerimeter} and \eqref{eqn:EstimateOnH2}, for any $t<T$ we obtain
    \begin{equation}
    \begin{split}
    \label{eq:chainreverse}
        |E_t| &\geq \int_{E_t\cap\{|\nabla w|>0\}} \frac{1}{|\nabla w|}|\nabla w| = \int_{-\infty}^t \left(\int_{\partial E_\tau\cap\{|\nabla w|>0\}} \frac{1}{|\nabla w|}\right)\de\tau \\
        &=\int_{-\infty}^t \left(\int_{\partial E_\tau} \frac{1}{|\nabla w|}\right)\de\tau \geq \int_{-\infty}^t \left(P( E_\tau)^{3/2}\left(\int_{\partial E_\tau}|\nabla w|^2\right)^{-1/2} \de \tau \right)\\
        &\geq \int_{-\infty}^t \frac{2\pi e^{3\tau/2}}{\sqrt{1+\frac{2}{3}\delta e^\tau}} \de \tau  \geq \int_{-\infty}^t \frac{2\pi e^{3\tau/2}}{\sqrt{1+\frac{2}{3}\delta e^T}} \de \tau \\
        &=
        \frac{1}{\sqrt{1+\frac{2}{3}\delta e^T}} \, \frac{4}{3}\pi e^{3t/2}=\frac{1}{\sqrt{1+\frac{2}{3}\delta e^T}}\, \frac{P(E_t)^{3/2}}{6\sqrt{\pi}}.
      \end{split}  
    \end{equation}
\end{proof}

In the $C^0_{\rm loc}$-asymptotically flat case we construct sets with arbitrarily large perimeter and volume that satisfy the reverse Euclidean isoperimetric inequality. This represents the crucial step for the proof of \cref{thm:CorMassa} and of \cref{thm:CorIsop}.
 
\begin{proposition}\label{prop:GrossiQuantoVoglio}
    Let $(M,g)$ be a $3$-dimensional complete $C^0$-Riemannian manifold without boundary, let $K\subset M$ be a compact set, and let $\Omega$ be an unbounded connected component of $M\setminus K$. Assume that $\Omega$ is $C^0_{\mathrm{loc}}$-asymptotic to $\mathbb R^3$ (\cref{def:C0locasymptotic}), and that $R_g\geq 0$ in the approximate sense on $\Omega\setminus K'$ (\cref{def:NonnegativeScalInApproximate}), where $K'\subset M$ is a compact set. Then, there exists a universal constant $\vartheta\in(0,1)$ such that the following holds.
    
    For every $\mathscr{P}>0$, there exists a set of finite perimeter $E\subset\subset \Omega\setminus K'$ such that
    \[
    \vartheta \mathscr{P}\leq P(E)\leq \mathscr{P},
    \]
    and
    \begin{equation}\label{eqn:IlControlloAlContrario}
    |E|\geq \frac{1}{6\sqrt{\pi}}\mathscr{P}^{3/2}\geq \frac{1}{6\sqrt{\pi}}P(E)^{3/2}.
    \end{equation}
\end{proposition}

\begin{proof}
Let $\rho>1$, $\eta<1/2$ to be chosen. The choice of $\rho$, only depending on $\mathscr{P}$ and on geometric constants on $\mathbb R^3$, will be made clear during the proof in \eqref{eqn:SceltaRho}. We do not insist on the precise choice of $\eta$ for the sake of readability; however, it will be clear from the proof that choosing $\eta<10^{-3}$ is sufficient. In this proof we will repeatedly use the elementary metric results recorded in \cref{lem:MetricheVicineImplicaBilipschitz}, and in the last part of the proof of \cref{lem:TechnicalLemmaC0Manifold}.

By applying a contradiction argument and \cref{rem:BiLipschitz} (see also the beginning of the proof of \cref{cor:C0locasymptoticPi}) we can find a compact set $\mathcal{C}\supset  (K\cup K')$ such that for every $x\in \Omega\setminus \mathcal{C}$ there exists $F_x:(B_{64\rho+64}(x),g)\to (\mathbb R^3,g_{\mathrm{eu}})$ which is a $(1+\eta)$-biLipschitz diffeomorphism with its image, with $F_x(x)=0$, and whose image contains $B_{32\rho+32}^{\mathbb R^3}(0)$.
Let us now fix $o$ such that 
\[
B_{64\rho+64}(o)\subset\subset \Omega\setminus \mathcal{C}.
\]
Let $g_i$ be a sequence of smooth Riemannian metrics on $M$ converging to $g$ locally uniformly with $R_{g_i} \ge -\eps_i$ on $\Omega\setminus\mathcal{C}$. We will frequently pass to subsequences with respect to $i$ in the course of the proof, without relabeling. We will denote $B_s(p),B_s^i(p)$ the open balls of radius $s$ and center $p$ in the metrics $g,g_i$, respectively. Up to passing to a subsequence with respect to $i$, we can assume that 
\begin{equation}\label{eqn:Contenimento}
B^i_{\rho-1}(o)\subset\subset B_\rho(o)\subset F_o^{-1}(B_{2\rho}^{\mathbb R^3}(0)) \subset\subset B_{4\rho+4}^i(o)\subset\subset \Omega\setminus \mathcal{C}\,\,\text{for all $i\in\mathbb N$}.
\end{equation}

Notice that (see \cref{def:C0locconvergence}) the compact $\mathcal{C}$ can be chosen such that, additionally, we have 
\begin{equation}\label{eqn:BelleVicine}
|(g-F_o^*g_{\mathrm{eu}})_x(v,v)|\leq \eta(F_o^*g_{\mathrm{eu}})_x(v,v),
\end{equation}
for every $x\in B_{64\rho+64}(o)$, and every $v\in T_xM$. Then, for $i$ large enough, we have
\begin{equation}\label{eqn:ErControllo}
|(g^i-F_o^*g_{\mathrm{eu}})_x(v,v)|\leq 2\eta(F_o^*g_{\mathrm{eu}})_x(v,v),
\end{equation}
for every $x\in B_{32\rho+32}^i(o)$, and every $v\in T_xM$. As a consequence, the map $F_x:(B^i_{16\rho+16}(x),g_i)\to (\mathbb R^3,g_{\mathrm{eu}})$ is $(1+3\eta)$-biLipschitz with its image. Notice also that the image of this map contains the ball $B_{8\rho+8}^{\mathbb R^3}(0)$.
\smallskip

As a consequence of \eqref{eqn:ErControllo}, if $\eta$ is small enough, the constants 
\[
C_{1,\mathrm{Sob}}(B^i_{4\rho+4}(o)),C_P(B^i_{4\rho+4}(o)),C_A(B^i_{4\rho+4}(o))
\]
appearing in \cref{thm:existenceimcfisoplocalized} (and defined in \eqref{eqn:pSobolevInequality}, \eqref{eqn:PoincareConstant}, \eqref{eqn:AhlforsConstant}), are uniformly bounded from above by a universal constant multiplied by the value of the corresponding constants on $\mathbb R^3$, which are independent of $\rho$. Moreover, since $F_x:(B^i_{16\rho+16}(x),g_i)\to (\mathbb R^3,g_{\mathrm{eu}})$ is $(1+3\eta)$-biLipschitz with its image, and contains the ball $B_{8\rho+8}^{\mathbb R^3}(0)$, the following holds: for every $r\leq \rho+1$ we can connect any two points $p,q\in \partial B^i_r(o)$ with a continuous curve $\gamma$ in the annulus $A^i_{3r/4, 5r/4}(o)$. Indeed, it suffices to connect $F_x(p)$ with $F_x(q)$ with a curve $\tilde\gamma\subset \mathbb R^3$ in the annulus $A^{\mathbb R^3}_{7r/8,9r/8}(0)$, and take $\gamma:=F_x^{-1}(\tilde\gamma)$. Moreover, with the same reasoning, the constant $C_{\mathrm{cov}}(B_{\rho+1}(o))$ defined in \eqref{eqn:CCov} is bounded above by a universal constant which only depends on the following universal constant independent of $\rho$:
\begin{equation}\label{eqn:CCovNuova}
\begin{split}
\tilde{C}\eqdef \min\Big\{N\in\mathbb N \st 
&\text{$A^{\mathbb R^3}_{r/2,3r/2}(0)$ is covered by $N$ open balls of radius $r/4$} \\
&\text{with centers in $A^{\mathbb R^3}_{r/2,3r/2}(o)$ for any $0<r\leq \rho+1$} \Big\}.
\end{split}
\end{equation}

Hence, we can apply \cref{thm:existenceimcfisoplocalized} on the ball $B^i_{2R}(o):=B^i_{4\rho+4}(o)$, where the radius $\rho+1$ in this proof corresponds to the number $\rho$ in the statement of \cref{thm:existenceimcfisoplocalized}. As a result of the discussion above, the constant $C$ appearing in \cref{thm:existenceimcfisoplocalized} when applied to the ball $B^i_{4\rho+4}(o)$ is bounded above by a universal constant $\xi>1$ independent of $i$, $\rho$, for $i$ large enough.

Now let $w_i$ be the weak IMCF issuing from $o$ in $B^i_{4\rho+4}(o)$, given by \cref{thm:existenceimcfisoplocalized}. Denote $E_t^i:=\{w_i< t\}$. Let $T_{\rho,\xi}:=2\log(\rho-1)-\xi-1$, and take $\rho>1$ such that
\begin{equation}\label{eqn:SceltaRho}
    4\pi e^{T_{\rho,\xi}}=4\pi e^{2\log(\rho-1)-\xi-1}=\mathscr{P}.
    \end{equation} 
    Notice that the choice of $\rho$ only depends on $\mathscr{P}$ and on the universal constant $\xi$.
    By the last assertion of \cref{thm:existenceimcfisoplocalized}, applied with $r:=\rho-1$, we have that 
\[
\{w_i< T_{\rho,\xi}\}=\Omega_{T_{\rho,\xi}}^i\subset B_{\rho-1}^i(o)\subset\subset B_\rho(o).
\]

Now we aim at applying item (3) of \cref{lem:Connectedness} with the choices $U:= B_{4\rho+4}^i(o)$, $\Omega:=F_o^{-1}(B_{2\rho}^{\mathbb R^3}(0))$, and $t_1:=T_{\rho,\xi}$. We stress that $\{w_i< T_{\rho,\xi}\}\subset\subset F_o^{-1}(B_{2\rho}^{\mathbb R^3}(0))$, and
\[
F_o^{-1}(B_{2\rho}^{\mathbb R^3}(0)),\,\,\text{and}\,\,\partial F_o^{-1}(B_{2\rho}^{\mathbb R^3}(0)) = F_o^{-1}(\partial B_{2\rho}^{\mathbb R^3}(0))\,\text{are connected}.
\]
Moreover, for any arbitrary $t_0<T_{\rho,\xi}$, one can define $\Omega''$ to be a sufficiently small ball so that all the hypotheses of item (2) in \cref{lem:Connectedness} are met, since we have $w_i(x)\to -\infty$ as $x\to o$. 
Finally noticing that $H_1(F_o^{-1}(B_{2\rho}^{\mathbb R^3}(0));\mathbb Z)=\{0\}$, one gets that all the hypotheses of item (2) and (3) in \cref{lem:Connectedness} are met, and thus $\partial\{w_i< t\}$ is connected for every $t<T_{\rho,\xi}$.

Therefore, recalling that $R_{g_i} \ge -\eps_i$ on $B^i_{4\rho + 4}(o)$, by \cref{lem:AlmostShi} we get
\begin{equation}\label{eqn:LimitNew}
    |\Omega_{T_{\rho,\xi}}^i|_i\geq \frac{1}{\sqrt{1+\frac{2}{3}\varepsilon_i e^{T_{\rho,\xi}}}}\frac{P_i(\Omega_{T_{\rho,\xi}}^i)^{3/2}}{6\sqrt{\pi}} = \frac{1}{\sqrt{1+\frac{2}{3}\varepsilon_i e^{T_{\rho,\xi}}}}\frac{(4\pi e^{T_{\rho,\xi}})^{3/2}}{6\sqrt{\pi}} = \frac{1}{\sqrt{1+\frac{2}{3}\varepsilon_i \frac{\mathscr{P}}{4\pi}}}\frac{\mathscr{P}^{3/2}}{6\sqrt{\pi}},
    \end{equation}
    where we used that $P_i(E_t^i)=4\pi e^t$, see \cref{prop:PerimetroEHawkingsProperty}, where $P_i(\cdot)$ denotes perimeter computed with respect to the metric $g_i$.

    Now, since $\Omega^i_{T_{\rho,\xi}}\subset\subset B_\rho(o)$, and the perimeters $P_i(\Omega^i_{T_{\rho,\xi}})$ are equibounded, we can use the precompactness and lower semicontinuity result in \cref{lem:PrecompactnessBV} to get a set $E_\rho\subset B_\rho(o)$ such that $\Omega^i_{T_{\rho,\xi}}\to E_\rho\subset\subset \Omega\setminus \mathcal{C}$ in $L^1$, $P(E_\rho)\leq \mathscr{P}$, and, by passing \eqref{eqn:LimitNew} to the limit, such that
    \begin{equation}\label{eqn:IlControlloDentro}
    |E_\rho|\geq \frac{\mathscr{P}^{3/2}}{6\sqrt{\pi}},
    \end{equation}
    which completes the proof of \eqref{eqn:IlControlloAlContrario}. 
    
    Finally, notice that $E_\rho \subset\subset B_{\rho+1}(o)$. Moreover \eqref{eqn:BelleVicine} holds. This implies that, if $\eta<1$ is small enough, on $B_{\rho+1}(o)$ there holds a $(1,1^*)$-Sobolev inequality \eqref{eqn:pSobolevInequality} with $C_{1, {\rm Sob}}(B_{\rho+1}(o))$ bounded from above by a universal constant only depending on the constant in the Euclidean $(1,1^*)$-Sobolev inequality. By \cref{rem:SobolevToIso}, this implies that $P(E)\geq \tilde{\vartheta}|E|^{2/3}$ for a universal $\tilde{\vartheta}$, for every $E\subset\subset B_{\rho+1}(o)$. Applying the latter inequality on $E_\rho$, and using again \eqref{eqn:IlControlloDentro}, we finally get $P(E_\rho)\geq \vartheta \mathscr{P}^{2/3}$ for a universal $\vartheta$, concluding the proof.
\end{proof}

\subsection{Proof of the main results and consequences}

We are now ready to prove the main results of the paper.

\begin{proof}[Proof of \cref{thm:CorMassa}]
It is a direct consequence of \cref{prop:GrossiQuantoVoglio} and the definition of quasi-local isoperimetric mass in \cref{def:IsopMass}.
\end{proof}

\begin{proof}[Proof of \cref{thm:small}]
The proof follows from a simplification of the arguments that led to \cref{prop:GrossiQuantoVoglio}. Indeed, we can apply \cref{lem:AlmostShi} to $B^i_{r(o)}(o)$, where $r(o)$ is chosen so that $B^i_{r(o)}(o)$ is contractible for any $i$ sufficiently large. As in the proof of \cref{prop:GrossiQuantoVoglio}, one can find a common time $T\in\R$ such that weak inverse mean curvature flows $w^i_{r(o)}$ in $B^i_{r(o)}(o)$ issuing from $o$ are well-defined with $\{w^i_{r(o)} \le T\} \Subset B^i_{r(o)}(o)$. Passing to the limit as $i \to +\infty$ suitable sublevel sets $\{w^i_{r(o)} \leq t\}$, for $t < T$, one gets as in \cref{prop:GrossiQuantoVoglio} limit sets $E_t$ satisfying the reverse isoperimetric inequality with sharp constant. By \eqref{eq:imcfproper}, one can also choose $t$ so that the sets $E_t$ are contained in any arbitrarily small ball centered at $o$.
\end{proof}

\begin{proof}[Proof of \cref{thm:CorIsop}]
We deal with the existence of isoperimetric sets of big volume first.
Suppose by contradiction that there exists $v_0 >0$ such that for every $v\in (v_0,+\infty)$ there are no isoperimetric sets of volume $v$ in $M$. We claim that then $I$ is strictly increasing on $(2v_0,+\infty)$.

Let $v\in (2v_0,+\infty)$. By \cref{thm:MassDecompositionC0} there is an isoperimetric set $E\subset M$ with $|E|\leq v_0$ (possibly empty) and a ball $B\subset \mathbb R^3$ with $|B|\geq v-v_0 \geq v/2$ such that 
\[
v=|E|+|B|_{\mathrm{eu}},\qquad I(v)=P(E)+P_{\mathrm{eu}}(B).
\]
Let now $\varepsilon>0$ be such $B_\varepsilon$ is a ball in $\mathbb{R}^3$ concentric to $B$ and with volume $|B_\varepsilon|_{\mathrm{eu}}=|B|_{\mathrm{eu}}-\varepsilon$. Notice that, by approximating $B_\varepsilon\subset \mathbb R^3$ with sets diverging along the manifold, arguing as in the proof of the upper bound for the isoperimetric profile in \cref{cor:Continuity}, we get $I(v-\varepsilon)\leq P(E)+P_{\mathrm{eu}}(B_\varepsilon)$. Thus, taking $\varepsilon\to0$, we find
\[
\frac{I(v)-I(v-\varepsilon)}{\varepsilon} \geq \frac{P_{\mathrm{eu}}(B)-P_{\mathrm{eu}}(B_\varepsilon)}{\varepsilon}
\xrightarrow[\eps\to 0]{}
 2 \left(\frac{4\pi}{3}\right)^{1/3}|B|_{\mathrm{eu}}^{-1/3}
 \geq 2 \left(\frac{4\pi}{3}\right)^{1/3}v^{-1/3}.
\]
    Thus we get that the lower left Dini derivative satisfies
    \[
    \overline{D^-} I(v):=\liminf_{\varepsilon\to 0^+} \frac{I(v)-I(v-\varepsilon)}{\varepsilon} \geq 2 \left(\frac{4\pi}{3}\right)^{1/3}v^{-1/3} 
    >0, \qquad \forall v\in (2v_0,+\infty).
    \]
    Since $I$ is continuous by \cref{cor:Continuity},  the latter implies that $I$ is strictly increasing on $(2v_0,+\infty)$. Thus the sought claim is proved.

    We aim now at showing that there exists an isoperimetric set with volume strictly greater than $2v_0$, thus reaching a contradiction. Fix $v>2v_0$. By \cref{thm:MassDecompositionC0} there is an isoperimetric set $E\subset M$ with $|E|\leq v_0$ and a ball $B\subset \mathbb R^3$ with $|B|\geq v-v_0>0$ such that 
\[
v=|E|+|B|_{\mathrm{eu}},\qquad I(v)=P(E)+P_{\mathrm{eu}}(B).
\]
By \cref{thm:BoundedIsoperimetricResidues} we know that $E$ is bounded. 

We apply \cref{prop:GrossiQuantoVoglio} with $K=\emptyset$ and $K'=\overline{B}_r(o)$ for some ball $B_r(o)$ with $r>1$ such that $E\cup \mathcal{C} \subset B_{r-1}(o)$, where $\mathcal{C}$ is as in the statement. Then there is a set $F\subset M$ such that $F\subset\subset M\setminus K'$, $P(F)\leq P_{\mathrm{eu}}(B)$, and 
\[
|F|\geq \frac{1}{6\sqrt{\pi}}P_{\mathrm{eu}}(B)^{3/2} = |B|_{\mathrm{eu}}.
\]
 Thus the set $E\cup F$ is such that 
\[
|E\cup F| = |E|+|F| \geq |E|+|B|_{\mathrm{eu}} = v.
\]
Hence, since $I$ is strictly increasing on $(2v_0,+\infty)$, one gets
\[
I(|E\cup F|) \geq I(v) = P(E)+P_{\mathrm{eu}}(B),
\]
but at the same time
\[
\begin{split}
I(|E\cup F|) &\leq P(E\cup F)=P(E)+P(F)\leq P(E)+P_{\mathrm{eu}}(B).
\end{split}
\]
Hence all the inequalities in the previous formula are equalities, and then $E\cup F$ is an isoperimetric set with volume $>2v_0$, which is a contradiction. 
\smallskip

We now briefly argue for the existence of isoperimetric sets with small volume. Assume again for a contradiction that there exists $v_0$ such that no isoperimetric set of volume $v$ exists for any $v \in (0, v_0)$. Then, by \cref{thm:MassDecompositionC0}, for every $v\in (0,v_0)$ we have $I(v) = P_{\rm eu}(B_v)$, where $B_v$ is the round ball in $(\R^3, \delta)$ of volume $v$. We can now reach a contradiction as follows. By \cref{thm:small} there is $v'\in (0,v_0)$ and $F$ such that 
\[
|F|=v', \qquad |F|\geq \frac{1}{6\sqrt{\pi}}P(F)^{3/2}.
\]
Thus, since for every $v\in (0,v_0)$ we have $I(v)=P_{\mathrm{eu}}(B_v)$, we can use the information above to conclude that 
\[
P(F)\geq I(v')= \sqrt[3]{36\pi} (v')^{2/3} \geq P(F).
\]
Hence, all the above inequalites are in fact equalities, and $F$ is an isoperimetric set with volume $v'\in (0,v_0)$, reaching a contradiction.
\end{proof}

In the proof of \cref{thm:CorIsop} we argued that, if for some $v_0 > 0$ no isoperimetric sets exist for any volume $v > v_0$  then 
the isoperimetric profile is strictly increasing for large volumes. This implied that some isoperimetric set of large volume exists, resulting in a contradiction. In fact, if the isoperimetric profile is strictly increasing, then one can deduce that isoperimetric sets exist for any volume. This is what happens in a smooth asymptotically flat $3$-manifold of nonnegative scalar curvature that is complete with no closed minimal surfaces or endowed with a horizon boundary. For this reason, the following result can be seen as a generalization of the existence result for any volume obtained in \cite[Proposition K.1]{CarlottoChodoshEichmair} (see also \cite[Theorem 3.6]{BFSurvey}). Moreover, when 
in addition the isoperimetric profile diverges at infinity, such isoperimetric sets realize the isoperimetric mass in the sense of \eqref{eq:isomassconisoperimetrici}.
The isoperimetric profile diverges at infinity for instance when a global Euclidean-like isoperimetric inequality is in force. The latter happens, for example, when the manifold is $C^0$-asymptotically flat globally.

\begin{proposition}
\label{prop:profileincreasing}
Let $(M, g)$ be a $C^0$-Riemannian manifold that is $C^0_{\mathrm{loc}}$-asymptotic to $\mathbb R^3$, and such that $R_g\geq 0$ in the approximate sense on $M\setminus K$, where $K\subset M$ is a compact set. Assume  that its isoperimetric profile $I(v)$ is strictly increasing. Then, for every volume $v$, there exists an isoperimetric set $E_v$ of volume $v$ on $M$. If in addition $\lim_{v \to + \infty} I(v) = + \infty$, then
\begin{equation}
\label{eq:isomassconisoperimetrici}
    \mathfrak{m}_{\mathrm{iso}} = \limsup_{v \to +\infty} \frac{2}{P(E_v)}\left(|E_v|-\frac{P(E_v)^{3/2}}{6\sqrt{\pi}}\right).
\end{equation}
\end{proposition}

\begin{proof}
We assume by contradiction that, for some $v > 0$, there exists no isoperimetric sets of volume $v$. By \cref{thm:MassDecompositionC0}, we have $I(v) = P(E) + P_{\rm eu} (B)$ for a possibly empty $E \subset M$ realizing $P(E) = I(|E|)$ and $B \subset \mathbb{R}^3$ a nonempty ball, such that $|E| + |B|_{\rm eu} = v$.  
The proof of \cref{thm:CorIsop} shows that we can find a set $F \subset\subset M\setminus \overline{B}_{r+1}(o)$, for some ball such that $E\subset B_r(o)$, with $P(F) \leq P_{\rm eu} (B)$ and with volume $|F| \geq |B|_{\rm eu}$. If $|F| = |B|_{\rm eu}$, then $E \cup F$ is an isoperimetric set of volume $v$, giving a contradiction. If $\abs{F} > \abs{B}_{\rm eu}$, we derive a contradiction with strict monotonicity of $I$. Indeed, on the one hand we would have
\[
I(\abs{E} + \abs{F}) > I(\abs{E} + |B|_{\rm eu}) = I(v),
\]
and on the other hand there holds
\[
I(\abs{E} + \abs{F}) \leq P(E) + P(F) \leq P(E) + P_{\rm eu}(B) = I(v).
\]
This proves the existence of isoperimetric sets of any volume.

We are left to prove \eqref{eq:isomassconisoperimetrici}. 
Observe that, since the isoperimetric profile diverges at infinity, the isoperimetric sets $E_v$ have perimeter diverging to infinity as $v \to +\infty$, and thus they are valid competitors in the definition \eqref{eq:isomass} of $\mathfrak{m}_{\mathrm{iso}}$. Hence
\[
{\mathfrak{m}}_{\mathrm{iso}}  \geq \limsup_{v \to +\infty} \frac{2}{P(E_v)}\left(|E_v|-\frac{P(E_v)^{3/2}}{6\sqrt{\pi}}\right).
\]
On the other hand, let $(F_j)_{j \in \N}$ be any other sequence of finite perimeter sets such that $P(F_j) \to +\infty$. Let $E_j$ be an isoperimetric set of volume $V_j = \abs{F_j}$. Then, the sequence $(E_{j})_{j \in \N}$ satisfies
\[
\frac{2}{P(F_j)}\left(|F_j|-\frac{P(F_j)^{3/2}}{6\sqrt{\pi}}\right) \leq \frac{2}{P(E_{j})}\left(|E_{j}|-\frac{P(E_j)^{3/2}}{6\sqrt{\pi}}\right),
\]
implying \eqref{eq:isomassconisoperimetrici}.
\end{proof}
We close this section with the proof of \cref{thm:quasilocalrigidity}.

\begin{proof}[Proof of \cref{thm:quasilocalrigidity}]
Let $o \in M$ and $B_R(o) \Subset \Omega$ with $H_1(B_R(o), \mathbb{Z}) = \{0\}$. Let $w$ be given by \cref{thm:existenceimcfisoplocalized}. By \eqref{eq:quasilocalrigidity}, any set compactly contained in $\Omega$ satisfies the Euclidean isoperimetric inequality. Hence \cref{lem:AlmostShi} implies that \eqref{eq:quasishi} holds with $\delta = 0$ and with the equality sign for all $t \in (-\infty, T)$ for some $T$, i.e., sublevel sets of the local weak IMCF from $o$ achieve equality in the Euclidean isoperimetric inequality.

Going back to the proof of \cref{lem:AlmostShi}, it follows that all the inequalities in \eqref{eq:chainreverse} are equalities for all $t \in (-\infty, T)$. This implies that the Hawking mass $\mathfrak{m}_H(\partial E_t) $ vanishes for all $t \in (-\infty, T)$. At this point, we can conclude following the argument in \cite[pp. 422-424]{HuiskenIlmanen}. For the reader's convenience, we briefly include the line of reasoning.

Using that the derivative of the Hawking mass \cite[Geroch Monotonicity Formula 5.8]{HuiskenIlmanen} vanishes (see also \cite[Theorem 5.5]{BenattiPludaPozzetta} for the version obtained through $p$-harmonic approximation, as in the present case), one gets that $\partial E_t$ is a smooth totally umbilical hypersurface with constant mean curvature for any $t \in (-\infty,T)$. Also, the hypersurfaces $\partial E_t$ are not minimal since $\int_{\partial E_t} H^2 = 16\pi$ as $\mathfrak{m}_H(\partial E_t)=0$. One can also deduce from the latter that the flow does not jump.

It follows that the hypersurfaces $\partial E_t$ yield a smooth foliation of a neighborhood of $o$. Denoting by $H(t)>0$ the constant mean curvature of $\partial E_t$, we can rewrite the metric $g$ on $\{w < T\}$ as $g|_{\{w < T\}} = H^{-2}(t)\mathrm{d} t \otimes \mathrm{d}t + g_{\partial E_t}$. Differentiating $g_{\partial E_t}$ with respect to $t$ and exploiting the total umbilicity of $\partial E_t$, the explicit value of $H^2(t)$, and the Euclidean asymptotic behavior of $g$ close to $o$, one obtains that $(\{w < T\},  g)$ is isometric to $\left( [0, S) \times \mathbb{S}^2, \mathrm{d}s \otimes \mathrm{d}s + r^2 g_{\mathbb{S}^2} \right)$, for some $S>0$, where $g_{\mathbb{S}^2}$ is the standard round metric on the sphere. Thus, $(\{w<T\},g)$ is flat, as desired. We refer the reader also to \cite[Proof of Theorem 1.2]{benatti2022minkowski} for additional details on this computation.
\end{proof}

\appendix

\section{Auxiliary results}

In this appendix we collect two technical results we used in the paper.

\begin{lemma}\label{lem:CoareaContinuousManifold}
Let $(M, g)$ be an $n$-dimensional complete $C^0$-Riemannian manifold. Then the following hold.
\begin{itemize}
    \item Let $f \in \Lip_{\rm loc}(M)$. Then $\lip f = |\nabla f|$ almost everywhere.

    \item Let $\Omega\subset M$ be an open set and suppose that $g_i$ is a sequence of smooth Riemannian metrics converging to $g$ uniformly on $\Omega$. Then for any $f \in L^1_{\rm loc}(\Omega, g)$ there holds
    \[
    \lim_i |D f|_i(\Omega) = |D f|(\Omega),
    \]
    where $|D f|_i$ denotes the total variation of $f$ as a function in $L^1_{\rm loc}(\Omega, g_i)$.

    \item For any $f \in \Lip_{\rm loc}(M)$, there holds $|D f|   =\lip f \,\haus^n = |\nabla f | \, \haus^n$. In particular
    \[
    |D \dist_{x_0} | = \haus^n,
    \]
    for any $x_0 \in M$, where $\dist_{x_0}$ denotes distance from $x_0$.
\end{itemize}
\end{lemma}

\begin{proof}
Let $f \in \Lip_{\rm loc}(M)$. Recall that the identity $\lip f = |\nabla f|$ almost everywhere readily follows on smooth Riemannian manifolds exploiting the exponential map. Fix $x \in M$ and let $g_i$ be a sequence of smooth Riemannian metrics converging to $g$ in $C^0$-sense on a neighborhood $A$ of $x$. We can write $(1-\eps_i)^2 g(v,v) \le g_i(v,v) \le (1+\eps_i)^2 g(v,v) $ for any tangent vector $v$ on $A$, for some $\eps_i\to0$. Denote by $\dist_i$ the distance function on $(A,g_i)$ defined by taking infimum of lengths of curves contained in $A$. For any $i$, let $y_j\in A$ such that $\lim_j\dist_i(x,y_j)=0$ and
\[
{\rm lip}_i f(x) = \lim_j \frac{|f(x)-f(y_j)|}{\dist_i(x,y_j)},
\]
where ${\rm lip}_i f$ denotes the slope of $f$ as a function in $(A, g_i)$.
For $j$ large, $\dist_i(x,y_j)$ is realized by a curve contained in $\Omega$. Hence $\dist_i(x,y_j) \ge (1-\eps_i) \dist(x,y_j)$ for any $j$ large. Thus 
\[
{\rm lip} f(x) \ge \limsup_j \frac{|f(x)-f(y_j)|}{\dist(x,y_j)} \ge  (1-\eps_i) \lim_j \frac{|f(x)-f(y_j)|}{\dist_i(x,y_j)} = (1-\eps_i) {\rm lip}_i f(x).
\]
A symmetric argument implies that $\lim_i \lip_i f(x) = \lip f(x)$. Since $\lip_i f= |\nabla^{g_i} f|_{g_i}$ almost everywhere, and norms of gradients clearly pass to the limit, then $\lip f(x) = \lim_i |\nabla^{g_i} f|_{g_i} = |\nabla f|$ almost everywhere.

Let now $f, g_i$ be as in the second item.
We can write again that $(1-\eps_i)^2 g(v,v) \le g_i(v,v) \le (1+\eps_i)^2 g(v,v) $ for any tangent vector $v$ on $\Omega$, for some $\eps_i\to0$. 
Denote by $\dist_i$ the distance function on $(\Omega,g_i)$ defined by taking infimum of lengths of curves contained in $\Omega$ and by $|D (\cdot)|_i$ the corresponding total variation.
Let $f_k\in \Lip_{\rm loc}(\Omega)$ be a sequence converging to $f$ in $L^1_{\rm loc}$ on $(\Omega,g)$ such that $|D f|(\Omega) = \lim_k \int_\Omega {\rm lip} f_k$.
As before, one estimates
\[
{\rm lip} f_k(x) \ge (1-\eps_i) {\rm lip}_i f_k(x).
\]
Therefore
\[
|D f|(\Omega) = \lim_k \int_\Omega {\rm lip} f_k  \ge (1-\tilde\eps_i) \liminf_k \int_\Omega {\rm lip}_i f_k \de \haus^n_{g_i} \ge (1-\tilde\eps_i)  |D f|_i(\Omega),
\]
for suitable $\tilde\eps_i\to 0$. Hence $|D f|(\Omega) \ge \limsup_i |D f|_i(\Omega)$. An analogous argument shows that $|D f|(\Omega) \le \liminf_i |D f|_i(\Omega)$.

Now let $f \in \Lip_{\rm loc}(M)$. For any $x_0 \in M$ there exist $r_0>0$ and a sequence of smooth metrics $g_i$ on $B_{r_0}(x_0)$ uniformly converging to $g$ on $B_{r_0}(x_0)$. The above argument also shows that ${\rm lip}_i f \to {\rm lip} f$ pointwise on $B_{r_0}(x_0)$, hence in $L^1_{\rm loc}$ on $(B_{r_0}(x_0), g)$, being $f$ locally Lipschitz. Using the second item, and since $|Df|_i={\rm lip}_i f \, \haus^n_{g_i}$ we have
\[
|D f|(B_{r_0}(x_0)) = \lim_i \int_{B_{r_0}(x_0)} {\rm lip}_i f \de \haus^n_{g_i} = \int_{B_{r_0}(x_0)} {\rm lip} f \de \haus^n.
\]
Taking into account the first item, the third one follows as well.
\end{proof}

Boundedness of isoperimetric sets usually follows from a suitable deformation lemma, see e.g. \cite[Theorem B.1]{AFP21}. It is not completely clear how to work out such an argument in the $C^0$-setting, yet one can slightly modify the argument, taking advantage of the asymptotic behavior imposed, to obtain the following weaker version, that still suffices for our aims.

\begin{theorem}\label{thm:BoundedIsoperimetricResidues}
  Let $(M, g)$ be an $n$-dimensional complete $C^0$-Riemannian manifold that is $C^0_{\mathrm{loc}}$-asymptotic to $\mathbb R^n$. Assume that for some $V > 0$ there exist $E\subset M$ and a nonempty Euclidean ball $B \subset \R^n$ such that $\abs{E} + \abs{B}_{\mathrm{eu}} = V$ and $I(V) = P(E) + P_{\mathrm{eu}}(B)$. Then $E$ has a bounded representative.
\end{theorem}

\begin{proof}
Without loss of generality, we can assume that $E$ has positive volume. Indeed, if $|E|=0$, then $\emptyset$ is a bounded representative.
Take $p\in M$. 
Let $V(r) = |{E\setminus B_r(p)}|$. Call $A(r) = P(E, M \setminus B_r(p))$. By coarea (see \cref{thm:Coarea}), and since by \cref{lem:CoareaContinuousManifold} we have $|D\dist_p|=\haus^n$, notice that $V'(r) = - P(B_r(p),E)$. 
Moreover, a Euclidean-like isoperimetric inequality holds for small volumes. Indeed, the proof of \cref{cor:C0locasymptoticPi} (compare also with item (1) in \cref{lem:TechnicalLemmaC0Manifold}) shows that there is a constant $\zeta$ such that $|B_r(p)|\geq \zeta r^3$ for every $r\in (0,1]$, and every $p\in M$. Thus \cite[Proposition 3.20]{APP} gives that  $I(v)\geq \xi v^{2/3}$ for every $v\in (0,\eta)$ for some $\eta,\xi>0$. Now arguing verbatim as after \cite[Equation (B.1)]{AFP21}, we can reduce ourselves to show that, for some constant $C$, there holds
\begin{equation}
\label{eq:claim}
A(r) \leq - V'(r) + CV(r),
\end{equation}   
for almost all large enough $r$. Assuming by contradiction that $V(r)>0$ for any $r>0$, the latter would result in a contradiction by ODE comparison on $V(r)$, see \cite[Theorem B.1]{AFP21}. To this aim, let $B^r$ be a smooth deformation of $B \subset\R^n$ of volume $\abs{B} + V(r)$ such that
\begin{equation}
\label{eq:relazione2}
    P_{\mathrm{eu}}(B^r) \leq P_{\mathrm{eu}}(B) + C V(r),
\end{equation}
where $C$ only depends on $B$.
Notice it is enough to take $B^r$ to be a ball containing $B$.
 Then,
 $\abs{E\cap B_r(p)} + \abs{B^r}_{\mathrm{eu}} = V$ and thus, by using the analogue of \cite[Proposition 3.2]{AFP21}, which can be proved analogously as in \cite{AFP21} (again exploiting that, arguing as in \cite[Lemma 2.17]{AFP21}, the isoperimetric profile is achieved by bounded sets), we get
\begin{equation}\label{eq:relazione}
    P(E) + P_{\mathrm{eu}}(B) = I(V) \le I( |E|- V(r)) + I_{\R^n}(|B|+V(r)) \le P(E\cap B_r(p)) + P_{\mathrm{eu}}(B^r).
\end{equation}
On the other hand, for almost every radius, we have (see \cite[Proposition 2.6]{APPV23}, which can be applied thanks to \cref{cor:C0locasymptoticPi}) that
\begin{equation}
\label{eq:relazione1}
    P(E\cap B_r(p)) \leq P(E) - P(E, M\setminus B_r(p)) + P(B_r(p),E).
\end{equation}
Plugging \eqref{eq:relazione2} and \eqref{eq:relazione1} into \eqref{eq:relazione}, we are left with
\eqref{eq:claim}. So the proof is concluded.
\end{proof}

\bigskip

\noindent \textbf{Data availability statement.} Not applicable.

\bigskip

\noindent \textbf{Conflicts of interest.} The authors declare no conflicts of interest.

\printbibliography

\end{document}